\documentclass[a4paper,10pt]{article}
\usepackage{amssymb,amsmath,amsthm,mathrsfs}
\usepackage{fullpage}
\usepackage{hyperref}
\usepackage{eucal}
\usepackage{color}
\usepackage{stackrel}
\usepackage{graphicx}
\newcommand{\ie}{\emph{i.e.}}
\newcommand{\eg}{\emph{e.g.}}
\newcommand{\etc}{\emph{etc}}
\newcommand{\cf}{\emph{cf.}}
\newcommand{\etal}{\emph{et al.}}
\newcommand{\Hyper}{\mathbb{H}}
\newcommand{\Real}{\mathbb{R}}
\newcommand{\Com}{\mathbb{C}}
\newcommand{\Nat}{\mathbb{N}}

\newcommand{\Sphere}{\mathbb{S}}
\newcommand{\sgn}{\mathop{\mathrm{sgn}}\nolimits}

\newcommand{\Dom}{\mathsf{D}}
\newcommand{\Ran}{\mathsf{R}}
\newcommand{\Ker}{\mathsf{N}}

\newcommand{\eps}{\varepsilon}
\newcommand{\sii}{L^2}

\newcommand{\der}{\mathrm{d}}
\newcommand{\Hilbert}{\mathcal{H}}
\newtheorem{Theorem}{Theorem}
\newtheorem{Lemma}{Lemma}
\newtheorem{Proposition}{Proposition}
\newtheorem{Corollary}{Corollary}
\newtheorem{Conjecture}{Conjecture}
\theoremstyle{definition}

\newtheorem{Remark}{Remark}

\newtheorem{Assumption}{Assumption}
\numberwithin{equation}{section}
\def\OMIT#1{}
\usepackage[normalem]{ulem}
\definecolor{DarkGreen}{rgb}{0,0.5,0.1} 

\newcommand\soutD{\bgroup\markoverwith
{\textcolor{DarkGreen}{\rule[.5ex]{2pt}{1pt}}}\ULon}
\newcommand\soutM{\bgroup\markoverwith
{\textcolor{blue}{\rule[.5ex]{2pt}{1pt}}}\ULon}
\newcommand{\Hm}[1]{\leavevmode{\marginpar{\tiny%
$\hbox to 0mm{\hspace*{-0.5mm}$\leftarrow$\hss}%
\vcenter{\vrule depth 0.1mm height 0.1mm width \the\marginparwidth}%
\hbox to
0mm{\hss$\rightarrow$\hspace*{-0.5mm}}$\\\relax\raggedright #1}}}

\begin{document}
%
\title{The abstract Birman--Schwinger principle and spectral stability}  
\author{Marcel Hansmann\,$^a$ \ and \ David Krej\v{c}i\v{r}{\'\i}k\,$^b$}	
\date{\small 
\begin{quote}
\emph{
\begin{itemize}
\item[$a)$] 
Fakult\"at f\"ur Mathematik, Technische Universit\"at Chemnitz,
Reichenhainer Strasse 41, 09107 Chemnitz, Germany;
marcel.hansmann@mathematik.tu-chemnitz.de.%
\item[$b)$] 
Department of Mathematics, Faculty of Nuclear Sciences and 
Physical Engineering, Czech Technical University in Prague, 
Trojanova 13, 12000 Prague 2, Czechia;
david.krejcirik@fjfi.cvut.cz.%
\end{itemize}
}
\end{quote}
28 October 2020}
\maketitle
\begin{abstract} 
\noindent
We discuss abstract Birman--Schwinger principles to study spectra of
self-adjoint operators subject to small non-self-adjoint 
perturbations in a factorised form. In particular, we extend and in part improve a classical result by Kato which ensures spectral stability.
As an application, we revisit known results for 
Schr\"odinger and Dirac operators in Euclidean spaces
and establish new results for Schr\"odinger operators 
in three-dimensional hyperbolic space.
%
%
\end{abstract}
%

\section{Introduction}
%

%
\subsection{Motivations}
The present paper has three purposes. The first is to develop an abstract version of 
the so-called Birman--Schwinger principle,
which is a well known tool from the theory of Schr\"odinger operators. 
It is customarily used to transfer a differential equation to an integral equation
and has been employed in many circumstances over the last half century
since the pioneering works of 
Birman~\cite{Birman_1961} and Schwinger~\cite{Schwinger_1967}. 
In recent years, the method has been revived  
in the context of spectral theory
of non-self-adjoint Schr\"odinger and Dirac operators with complex potentials
as a replacement of unavailable variational techniques
(see, \eg, \cite{Frank_2011,
Cuenin-Laptev-Tretter_2014,Enblom_2016,Frank-Simon_2017,
Cuenin_2017,FKV,Enblom_2018,Cuenin-Siegl_2018,
FK9,IKL,Ibrogimov-Stampach_2019,CIKS}
to quote just a couple of most recent works).
While its usefulness is very robust, 
the method is usually applied to concrete problems \emph{ad hoc}
and not always rigorously. 
Here we suggest an abstract machinery directly applicable to concrete problems. Abstract versions of the Birman--Schwinger principle have been discussed before (see Remark \ref{Remark.literature} below), but this was usually restricted to (discrete) eigenvalues. In contrast, we also cover eigenvalues embedded in the essential spectrum as well as residual, continuous and essential spectra.

Our second goal is to use our abstract machinery to prove spectral stability given uniform bounds on the Birman--Schwinger operator. In particular, we will be able to derive such results without any smoothness assumptions (in the sense of Kato \cite{Kato_1966}) and will thus be able to extend and improve upon Kato's classical result (Theorem \ref{Thm.main.similar} below) on this topic.

Our third and final goal is to show the applicability of the abstract Birman--Schwinger principles. This will be illustrated via some known spectral enclosures
for Schr\"odinger and Dirac operators in Euclidean spaces, which we recover, and via a completely new result, namely the spectral stability for Schr\"odinger operators in three-dimensional hyperbolic space.

\subsection{Assumptions and notations} 
Throughout this paper $\Hilbert$ and $\Hilbert'$ denote complex separable Hilbert spaces and $\mathscr{B} (\Hilbert, \Hilbert')$ denotes the space of bounded linear operators from $\Hilbert$ to $\Hilbert'$. As usual, we set 
$\mathscr{B}(\Hilbert) := \mathscr{B}(\Hilbert, \Hilbert)$, \etc. 
We denote the inner product (which is linear in the second component) and norm in $\Hilbert$ as well as in $\Hilbert'$ by the same symbols, 
namely $(\cdot,\cdot)$ and $\|\cdot\|$, respectively. The latter is also used to denote the operator norms in $\mathscr{B}(\Hilbert,\Hilbert'), \mathscr{B}(\Hilbert)$ and so on. The particular meaning of each symbol should always be clear from the context. 
We denote the domain, kernel,
range and adjoint of an operator $A$ from $\Hilbert \to \Hilbert'$ by 
$\Dom(A)$, $\Ker(A)$, $\Ran(A)$ and $A^*$, respectively. Recall that the \emph{spectrum} $\sigma(H)$ of any closed operator~$H$ in~$\Hilbert$  
is the set of those complex numbers~$\lambda$ for which
$H-\lambda:\Dom(H)\to\Hilbert$ is not bijective.
The \emph{resolvent set} is the complement $\rho(H):=\Com\setminus\sigma(H)$.
The \emph{point spectrum} $\sigma_{\mathrm{p}}(H)$ of~$H$ 
is the set of eigenvalues of~$H$
(\ie~the operator $H-\lambda$ is not injective).
For the surjectivity, one says that 
$\lambda \in \sigma(H)$ belongs to the \emph{continuous spectrum}
$\sigma_{\mathrm{c}}(H)$
(respectively, \emph{residual spectrum} $\sigma_{\mathrm{r}}(H)$) of~$H$ 
if $\lambda\not\in\sigma_\mathrm{p}(H_V)$
and the closure of the range of $H-\lambda$ equals~$\Hilbert$  
(respectively, the closure is a proper subset of~$\Hilbert$).
Finally, we say that $\lambda \in \Com$ belongs 
to the \emph{essential spectrum} $\sigma_{e}(H)$ of~$H$ 
if~$\lambda$ is an eigenvalue of infinite geometric multiplicity
or the range of $H-\lambda$ is not closed.

\medskip
Our standing hypotheses are as follows.
\begin{Assumption}\label{Ass.Ass}
$H_0$ is a self-adjoint operator 
in $\Hilbert$ and  $|H_0| := (H_0^2)^{1/2}$. Moreover, $A : \Dom(A) \subset \Hilbert \to \Hilbert'$ and $B : \Dom(B) \subset \Hilbert \to \Hilbert'$ are linear operators such that $\Dom(|H_0|^{1/2}) \subset \Dom(A) \cap \Dom(B)$. We assume that for some (hence all) $b>0$
\begin{equation}
  \label{Ass.bounded}
  A(|H_0|+b)^{-1/2} \in \mathscr{B}(\Hilbert, \Hilbert'), \quad  B(|H_0|+b)^{-1/2} \in \mathscr{B}(\Hilbert, \Hilbert').
\end{equation}
Next, we set $G_0:=(|H_0|+1)$  and introduce the \emph{Birman--Schwinger operator}
\begin{equation}
  \label{Birman_Schwinger1} 
  K(\lambda):=  \big[AG_0^{-1/2}\big]\big[G_0(H_0-\lambda)^{-1}\big]
  \big[BG_0^{-1/2}\big]^* \in \mathscr{B}(\Hilbert'), \qquad \lambda \in \rho(H_0).
\end{equation}
Our final assumption is that there exists $\lambda_0 \in \rho(H_0)$ such that 
\begin{equation}
  \label{Ass.spectrum}
  -1 \notin \sigma(K(\lambda_0)).
\end{equation}
\end{Assumption}
\begin{Remark}\label{rem1}
  We note that $K(\lambda)$ is a bounded extension of 
  the (maybe more familiar) Birman--Schwinger operator $A(H_0-\lambda)^{-1}B^*$, defined on $\Dom(B^*)$. In particular, if $\Dom(B^*)$ is dense in $\Hilbert'$, then $K(\lambda)= \overline{ A(H_0-\lambda)^{-1}B^*}$. For instance, the latter is true if $B$ is closable. Moreover, setting $G_\delta:= (|H_0|+1+\delta)$ for $\delta > -1$, we note that we also have that  
  \begin{equation}\label{Birman_Schwinger2}
 K(\lambda) = \big[AG_\delta^{-1/2}\big]\big[G_\delta(H_0-\lambda)^{-1}\big]
 \big[BG_\delta^{-1/2}\big]^*,
\end{equation}
as follows from the fact that functions of $H_0$ commute (taking the respective domains into account) and
\[ 
\big[BG_0^{-1/2}\big]^*
= \big[(BG_\delta^{-1/2})(G_\delta^{1/2} G_0^{-1/2})\big]^* 
= \big[G_\delta^{1/2} G_0^{-1/2}\big]^* \big[BG_\delta^{-1/2}\big]^* 
= \big[G_\delta^{1/2} G_0^{-1/2}\big] \big[BG_\delta^{-1/2}\big]^*.
\] 
\end{Remark}

\begin{Remark}\label{Remark.literature}
There exist a variety of approaches to the Birman--Schwinger principle for factorable perturbations of a given (self-adjoint) operator $H_0$, \ie\ for a suitable closed extension $H_V$ of $H_0+B^*A$. Let us mention Kato's pioneering work \cite{Kato_1966} and the work by Konno and Koroda \cite{konno_finiteness_1966}. As some of the more recent articles on the topic we mention works by Gesztesy \etal\ \cite{gesztesy_nonselfadjoint_2005}, Latushkin and Sukhtayev \cite{latushkin_algebraic_2010}, Frank \cite{frank_eigenvalue_2018} 
and of Behrndt, ter Elst and Gesztesy \cite{behrndt_generalized_2020}. The assumptions on $A, B$ and $H_0$ made in these works are not uniform but vary from paper to paper. Our own assumptions take an intermediate position. For instance, we do not assume that $A$ or $B$ are closed, which is important for some applications. An example is provided, \eg, in the remark in Appendix B of \cite{frank_eigenvalue_2018}, which also shows that it can be advantageous to allow for the case $\Hilbert' \neq \Hilbert$. 

On the other hand, we do assume that $H_0$ is self-adjoint (but not necessarily bounded below), which some of the mentioned papers don't, and we do assume that $\Dom(|H_0|^{1/2}) \subset \Dom(A) \cap \Dom(B)$. The latter assumption allows for a quite explicit description of $H_V$ via quadratic forms (see Section \ref{Sec.pseudo}). In contrast to this, the weaker assumption that $\Dom(H_0) \subset \Dom(A) \cap \Dom(B)$ made in some of the mentioned papers would allow to define $H_V$ only implicitly via an associated resolvent equation (see (\ref{2nd})). 
\end{Remark} 

While Assumption (\ref{Ass.bounded}) is usually easy to verify in
concrete applications (for instance, it is certainly true if $A$ and
$B$ are closed as follows from the closed graph theorem), the direct
verification of (\ref{Ass.spectrum}) might not be that easy. For this
reason, the next lemma discusses some sufficient conditions for
(\ref{Ass.spectrum}) which might be easier to verify. 
\begin{Lemma}\label{Lemma1}
  Assume \eqref{Ass.bounded}. Then assumption
  \eqref{Ass.spectrum} is satisfied if one of the following three conditions holds:
  \begin{enumerate}
  \item[\emph{(i)}] 
  there exists $\lambda_0 \in \rho(H_0)$ such that $\|K(\lambda_0)\| < 1$,
  \item[\emph{(ii)}] 
  there exists $a \in (0,1)$ and $\delta> 0$ such that
  \begin{equation}
    \label{eq:3}
    \big\| \big[B(|H_0|+\delta)^{-1/2}\big]^*
    \big[A(|H_0|+\delta)^{-1/2}\big]\big\| \leq a,
  \end{equation}
  \item[\emph{(iii)}] 
  there exists $a \in (0,1)$ and $\delta> 0$ such that
    \begin{equation}
|(B\phi, A \psi)| \leq a \,
\big\| (|H_0|+\delta)^{1/2}\phi\big\| 
\big\| (|H_0|+\delta)^{1/2}\psi\big\|, 
\qquad \phi,\psi \in \Dom(|H_0|^{1/2}).\label{eq:1}
\end{equation}
  \end{enumerate}  
Moreover, the assumptions \eqref{Ass.bounded} and
  \eqref{Ass.spectrum} are \emph{both} satisfied if
\begin{enumerate}
\item[\emph{(iv)}	] 
there exists $a \in (0,1)$ and $b> 0$ such that 
  \begin{equation}
  \max\big( \|A\psi\|^2, \|B\psi\|^2\big) \leq a \, 
  \big\||H_0|^{1/2}\psi\big\|^2 + b \, \|\psi\|^2, \qquad \psi \in \Dom(|H_0|^{1/2}).\label{eq:2}
\end{equation}
In addition, if $A$ and $B$ are closed then it is sufficient 
that \eqref{eq:2} holds for $\psi \in \mathcal D$ where $\mathcal D$ is a core of $\Dom(|H_0|^{1/2})$.
  \end{enumerate}  
\end{Lemma}
\begin{proof}
  (i) follows from the fact that the spectral radius is dominated by the operator norm. For (ii) we first note that in view of (\ref{Birman_Schwinger2}) for $ \delta > 0$ we have 
\[ 
K(\lambda)= \big[A(|H_0|+\delta)^{-1/2}\big]
\big[(|H_0|+\delta)(H_0-\lambda)^{-1}\big]
\big[B(|H_0|+\delta)^{-1/2}\big]^*, 
\qquad \lambda \in \rho(H_0).
\]  
Since for two bounded operators $C,D$ we have $\sigma(CD)\setminus \{0\} = \sigma(DC) \setminus \{0\}$, we thus obtain that $-1 \notin \sigma(K(\lambda))$ 
if, and only if,
  \[ 
  -1 \notin \sigma\Big(
  \big[(|H_0|+\delta)(H_0-\lambda)^{-1}\big]
  \big[B(|H_0|+\delta)^{-1/2}\big]^*
  \big[A(|H_0|+\delta)^{-1/2}\big]
  \Big)
  \]
and the last condition is satisfied if the norm of the operator 
on the right-hand side is smaller than one. 
But by assumption there exist $a,\delta>0$ such that  
$
  \big\|
  \big[B(|H_0|+\delta)^{-1/2}\big]^* \big[A(|H_0|+\delta)^{-1/2}\big]
  \big\| \leq a < 1
$,  hence it suffices to choose $\lambda \in \rho(H_0)$ such that  
$
  \big\|(|H_0|+\delta)(H_0-\lambda)^{-1}\big\| \leq 1/a
$. 
The latter is satisfied if $\lambda=i \eta$ with $\eta > 0$ sufficiently large, which concludes the proof of~(ii). Continuing, we note that (iii) follows from (ii) since 
\begin{align*}
& & \big\|\big[B(|H_0|+\delta)^{-1/2}\big]^*A(|H_0|+\delta)^{-1/2}\big\| \leq a &
\\
\Leftrightarrow & & \big\|\big[A(|H_0|+\delta)^{-1/2}\big]^*
B(|H_0|+\delta)^{-1/2}\big\| \leq a &
\\
\Leftrightarrow &  \qquad \forall f, g \in \Hilbert:  
& \big| ( B(|H_0|+\delta)^{-1/2} f, A(|H_0|+\delta)^{-1/2} g) \big| 
\leq a \, \|f\| \|g\| & 
\\
\Leftrightarrow   & \qquad \forall \phi, \psi \in \Dom(|H_0|^{1/2}): & | ( B \phi  , A \psi) | \leq a \, \big\|(|H_0|+\delta)^{1/2}\phi\big\| 
\big\|(|H_0|+\delta)^{1/2}\psi\big\| &.
\end{align*}
Concerning (iv) we note that given (\ref{eq:2}), for $\phi \in \Hilbert$ and $\psi = ( |H_0|+\delta)^{-1/2}\phi$, where $\delta = b/a$, we obtain that 
\begin{align*}
&    \big\| A( |H_0|+\delta)^{-1/2}\phi\big\|^2 
\leq a \, \big\||H_0|^{1/2}\psi\big\|^2 + b \|\psi\|^2 
= a \, \big\| (|H_0|+\delta)^{1/2}\psi\big\|^2
= a \, \|\phi\|^2.
  \end{align*}
Hence $A( |H_0|+\delta)^{-1/2}$ is bounded and $\big\| A( |H_0|+\delta)^{-1/2}\big\| \leq \sqrt{a}$ 
and the same is true of $B( |H_0|+\delta)^{-1/2}$ and its norm, so
\eqref{Ass.bounded} is satisfied. Moreover, the validity of
\eqref{Ass.spectrum} follows from~(ii), the submultiplicativity of the operator norm and the fact that the norm of a bounded operator and its adjoint coincide.  Finally, concerning the last statement of (iv)  
we note that in case~$A$ and~$B$ are closed, the estimate (\ref{eq:2}) will hold for all $\phi, \psi \in \Dom(|H_0|^{1/2})$ once they hold for $\phi, \psi$ 
in a core of~$|H_0|^{1/2}$. 
\end{proof}

The composition $V := B^*A$ (with its natural domain) is a well defined operator in~$\Hilbert$. However,  since~$H_0$ is not necessarily bounded from below,
the machinery of closed sectorial forms 
and the customary Friedrichs extension 
of the operator sum $H_0 + V$
are not available to us. As a replacement, below we will introduce a unique closed extension~$H_V$ of $H_0 + V$
by means of the so-called \emph{pseudo-Friedrichs extension} 
\cite[Sec.~VI.3.4]{Kato} (see Section~\ref{Sec.pseudo} for more details). First, however, let us discuss our main results about this operator.

\subsection{Our main results}
The well known version of the \emph{Birman--Schwinger principle} 
is formulated by the following equivalence.
\begin{Theorem}\label{Thm.evs0} 
Suppose Assumption~\ref{Ass.Ass}. Then
\begin{equation}\label{BS.principle}
  \forall \lambda \in \Com \setminus \sigma(H_0) 
  \,, \qquad
  \lambda \in \sigma_{\mathrm{p}}(H_V) 
  \quad \Longleftrightarrow \quad
  -1 \in \sigma_{\mathrm{p}}(K_\lambda)
  \,.
\end{equation}
\end{Theorem}

We establish the validity of this equivalence 
in the fully abstract setting above 
(Theorems~\ref{Thm.evs1} and~\ref{Thm.evs2}). 
While in slightly different settings this result has been proved a variety of times before (see, \eg, the papers cited in Remark~\ref{Remark.literature}), one of the main points of the present paper is that suitably adapted versions of 
the Birman--Schwinger principle hold also for: 
\begin{itemize}
\item
\emph{all} eigenvalues 
$\sigma_\mathrm{p}(H_V) \setminus \sigma_\mathrm{p}(H_0)$;
\hfill
(Theorem~\ref{Thm.evs})
\item
residual spectrum $\sigma_\mathrm{r}(H_V) \setminus \sigma_\mathrm{p}(H_0)$;
\hfill
(Theorems~\ref{Thm.res0} and~\ref{Thm.res})
\item
essential spectrum $\sigma_\mathrm{e}(H_V) \setminus \sigma(H_0)$.
\hfill
(Theorem~\ref{Thm.cont})
\end{itemize}   
Such variants of the Birman--Schwinger principle seem to be less known.
An exception is \cite{FKV} in which Fanelli, Vega 
and one of the present authors established results
of this type in the case of Schr\"odinger operators
(see also \cite{Davies-Nath_2002,Pushnitski_2011,Frank-Simon_2017}).

Using the Birman--Schwinger operator and the Birman--Schwinger principle,
we establish stability results about the spectrum of~$H_V$, assuming that $K_z$ is uniformly bounded in $z$, \ie, 
\begin{equation}\label{Ass.uniform}
 \sup_{ z \in \rho(H_0) } \|K_z\| < \infty.  
\end{equation}
The first of our main results in this direction 
is the following theorem.
\begin{Theorem}\label{Thm.main0}
Suppose Assumption~\ref{Ass.Ass} and \eqref{Ass.uniform}. 
Then $\sigma(H_0) \subset \sigma(H_V)$.
\end{Theorem}
\begin{Remark}
 It is clear that the conclusion of Theorem \ref{Thm.main0} is generally false if (\ref{Ass.uniform}) is not satisfied. Just consider the case where $A=I$ and $B=i \cdot I$, where $\sigma(H_V)=\sigma(H_0)+i$, 
 \ie\ the spectrum of $H_0$ is shifted into the complex plane.
\end{Remark}

From our point of view, the remarkable thing about Theorem~\ref{Thm.main0} is that it holds without any smallness assumption on $\sup_z \|K_z\|$. Indeed, in all applications of the Birman--Schwinger principle to spectral estimates that we are aware of one assumes that $\sup_z \|K_z\|$ is sufficiently small and then derives information about $\sigma(H_V)$. The fact that some information can also be obtained without assuming that the supremum is small seems to have been completely overlooked so far.  A possible reason for this might be that in typical applications the spectrum of $H_0$ is purely essential and the assumption (\ref{Ass.uniform}) usually implies that the resolvent difference of $H_0$ and $H_V$ is compact, hence 
$\sigma(H_0)= \sigma_\mathrm{e}(H_0)=\sigma_\mathrm{e}(H_V)$. 
In general, however, there is no reason to believe that (\ref{Ass.uniform}) should imply such a compactness property.

In case that $\|K_z\|$ is indeed uniformly small, \ie
\begin{equation}\label{Ass.small1}
  \exists c<1 :
 \sup_{ z \in \rho(H_0) } \|K_z\| \leq c,
\end{equation} 
one obtains much stronger information on $\sigma(H_V)$.
\begin{Remark}
  Let us note that given (\ref{Ass.small1}) the Assumption \ref{Ass.Ass} reduces to
  (\ref{Ass.bounded}) since (\ref{Ass.spectrum}) is automatically
  satisfied as we discussed in Lemma \ref{Lemma1}.
\end{Remark}

\begin{Theorem}\label{Thm.main} 
Suppose Assumption~\ref{Ass.Ass} and \eqref{Ass.small1}. 
Then the following holds:
\begin{enumerate}
\item[\emph{(i)}] $\sigma(H_0) = \sigma(H_V)$.
\item[\emph{(ii)}] 
  $[\sigma_\mathrm{p}(H_V) \cup \sigma_\mathrm{r}(H_V) ] \subset \sigma_\mathrm{p}(H_0)$ and $\sigma_\mathrm{c}(H_0) \subset \sigma_\mathrm{c}(H_V)$.

  In particular, if $\sigma(H_0)=\sigma_\mathrm{c}(H_0)$, then $\sigma(H_V)=\sigma_\mathrm{c}(H_V)=\sigma_\mathrm{c}(H_0)$.
\end{enumerate}
\end{Theorem}

So the spectra of~$H_V$ and~$H_0$ coincide if the perturbation~$V$ is small
in the sense of~\eqref{Ass.small1}.   
Moreover, the spectrum of~$H_V$ is purely continuous
if it is the case of~$H_0$.
As we will see, these stability properties follow directly from Theorem \ref{Thm.main0} and from Theorem~\ref{Thm.evs0} 
and its variants mentioned below it. 
\begin{Remark}
 It is well known that $\sigma(H_0)=\sigma(H_V)$ need not be true if $\sup \|K_z\| \geq 1$, see \eg\ the proof of the $d=3$ case of Theorem 2 in \cite{Frank_2011}. 
\end{Remark}

We do not know whether in general, given (\ref{Ass.small1}), the continuous-, point- and residual spectra of $H_0$ and $H_V$ coincide. However, this is the case if $A$ is \emph{relatively smooth} with respect to \emph{$H_0$}, which means that $A: D(A) \subset \Hilbert \to \Hilbert'$ is closed with $D(H_0) \subset D(A)$ and 
\begin{equation}
  \label{Ass.smooth}
  \sup_{z \in \Com \setminus \Real, \ \psi \in \Hilbert \setminus \{0\}} |\Im(z)| \cdot \|A(H_0-z)^{-1}\psi\|^2/\|\psi\|^2 < \infty.
\end{equation}
The notion of relative smoothness is due to Kato \cite{Kato_1966} and we should note that there exist several equivalent ways to introduce this concept.

\begin{Corollary}\label{Cor1}
Suppose Assumption~\ref{Ass.Ass} and \eqref{Ass.small1}. 
Moreover, assume that $A$ is relatively smooth with respect to $H_0$. Then
\[ \sigma_\mathrm{c}(H_V)=\sigma_\mathrm{c}(H_0), \qquad \sigma_\mathrm{p}(H_V)=\sigma_\mathrm{p}(H_0) \qquad \text{and} \qquad \sigma_\mathrm{r}(H_V)=\sigma_\mathrm{r}(H_0)=\emptyset.\]
\end{Corollary}
\begin{proof}[Proof of Corollary \ref{Cor1}]
 In view of Theorem \ref{Thm.main} it is sufficient to show that $\sigma_\mathrm{p}(H_0) \subset \sigma_\mathrm{p}(H_V)$. So suppose that for some $\lambda \in \Real$ and $\psi \in \Dom(H_0) \setminus \{0\}$ we have $H_0\psi = \lambda \psi$. Then for $\eps > 0$ we also have $-i\eps (H_0-\lambda-i\eps)^{-1}\psi = \psi$ and hence 
\[  \|A(H_0-\lambda-i\eps)^{-1}\psi\| = \eps^{-1} \|A\psi\| .\]
Since this is true for all $\eps > 0$, assumption (\ref{Ass.smooth}) implies that $A\psi =0$. We will see below that $H_V$ is a closed extension of $H_0 + B^*A$, so we obtain that $\psi \in \Dom(H_V)$ and $H_V \psi = H_0 \psi = \lambda \psi$.  
\end{proof}

\begin{Remark} \label{Rem.smooth}
Even if $A$ and $B$ are closed and satisfy Assumption \ref{Ass.Ass} and (\ref{Ass.small1}), this does not imply that $A$ is $H_0$-smooth. For instance, the mentioned assumptions on $A$ and $B$ are satisfied if $B=0$ and $A$ is  \emph{any} closed operator from $\Hilbert \to \Hilbert'$ with $D(|H_0|^{1/2}) \subset D(A)$.
\end{Remark}

There is one important case where (\ref{Ass.small1}) does imply smoothness of $A$ with respect to $H_0$, namely if $A=DB$ for some $D \in \mathscr{B}(\Hilbert')$. This leads to another corollary of Theorem \ref{Thm.main}.

\begin{Corollary}\label{Cor2}
Suppose Assumption~\ref{Ass.Ass} and~\eqref{Ass.small1}. Moreover, suppose that $A$ is closed and that $A=DB$ for some  $D \in \mathscr{B}(\Hilbert')$. Then   
\[ \sigma_\mathrm{c}(H_V)=\sigma_\mathrm{c}(H_0), \qquad \sigma_\mathrm{p}(H_V)=\sigma_\mathrm{p}(H_0) \qquad \text{and} \qquad \sigma_\mathrm{r}(H_V)=\sigma_\mathrm{r}(H_0)=\emptyset.\]
\end{Corollary}
\begin{proof}[Proof of Corollary \ref{Cor2}]
By \cite[Thm.~5.1]{Kato_1966} the $H_0$-smoothness of $A$ is equivalent to the fact that 
\begin{equation}
 \sup_{z \in \Com \setminus \Real, \ \psi \in \Dom(A^*) \setminus \{0\}} 
 \big| \big( 
 [ (H_0-z)^{-1}-(H_0-\overline z)^{-1}] A^*\psi, A^* \psi 
 \big) \big| / \|\psi\|^2 < \infty.\label{eq:5}
\end{equation}
But using our assumptions, for $z \in \Com \setminus \Real$ and $\psi \in \Dom(A^*)$ with $\|\psi\|=1$ we can estimate
\begin{align*}
| ( (H_0-z)^{-1} A^*\psi, A^* \psi ) | &=  | (A(H_0-z)^{-1} B^*D^*\psi, \psi)| \leq \|K_z\| \|D\| \leq \|D\| \sup_ {z \in \Com \setminus \Real} \|K_z\|.
\end{align*}
Using the same inequality to estimate the second term of the difference in (\ref{eq:5}) we see that the left-hand side of (\ref{eq:5}) is indeed finite. Now apply Corollary \ref{Cor1}. 
\end{proof}

In order to put Theorem \ref{Thm.main} and its corollaries into perspective, we need to take a closer look at Kato's classical work \cite{Kato_1966}. We do this in the following section.

\subsection{Kato's results}
The main result of Kato's 1966 paper \cite{Kato_1966} is the following theorem.

\begin{Theorem}%
[{\cite[Thm.~1.5]{Kato_1966}}]\label{Thm.main.similar}
Let $H_0$ be self-adjoint in $\Hilbert$ and suppose that  $A,B$ are closed operators from~$\Hilbert$ to~$\Hilbert'$ 
with $\Dom(H_0) \subset \Dom(A) \cap \Dom(B)$ which are smooth relative to $H_0$. Moreover, suppose that there exists $c < 1$ such that
\begin{equation}
\sup_{ z \in \Com \setminus \Real} \|A(H_0-z)^{-1}B^*\| \leq c.\label{eq:4}
\end{equation}
Then there exists a closed extension $\tilde{H}_V$ of $H_0+B^*A$ which is similar to $H_0$ (so in particular, the continuous, point and residual spectra of $\tilde{H}_V$ and $H_0$ coincide). Moreover, the operator $\tilde{H}_V$ satisfies the generalised second resolvent equation 
\begin{equation}\label{2nd}
  \forall \xi\in\Com\setminus\Real \,, \qquad
  (\tilde{H}_V-\xi)^{-1} - (H_0-\xi)^{-1}
  = - \overline{(H_0-\xi)^{-1}B^*} A (\tilde{H}_V-\xi)^{-1}
  \,.
\end{equation} 
\end{Theorem}

\begin{Remark}
Here the similarity means that there exists an operator
$W \in \mathscr{B}(\Hilbert)$ such that $W^{-1} \in \mathscr{B}(\Hilbert)$
and $\tilde{H}_V=WH_0W^{-1}$. In other words, $\tilde{H}_V$~is \emph{quasi-self-adjoint} (\cf~\cite{KS-book}). We note that Kato actually states his theorem for the more general case that $H_0$ is closed and densely defined with $\sigma(H_0) \subset \Real$. 
\end{Remark}

To compare Kato's result with our results of the previous section, one first needs to check that his operator $\tilde{H}_V$ and our pseudo-Friedrichs extension $H_V$ (to be constructed below) do indeed coincide if Assumption~\ref{Ass.Ass} is satisfied. This will be done in the Appendix (Proposition \ref{Prop.Appendix}) under the additional assumption that $\Dom(A)=\Dom(B)=\Dom(|H_0|^{1/2})$.

Now let us start with a comparison of the assumptions of Kato and of our results above. First, we note that Kato requires the operators $A$ and $B$ to be closed, which we don't, but that he doesn't assume that $\Dom(|H_0|^{1/2}) \subset \Dom(A) \cap \Dom(B)$, which we do. Second, we note that given Kato's assumptions, the operator $A(H_0-z)^{-1}B^*$ is just the closure of our Birman--Schwinger operator $K(z)$, 
so assumption~(\ref{eq:4}) 
is the same as our assumption~(\ref{Ass.small1}). In particular, let us emphasise that Kato does not provide any conclusions under the weaker assumption (\ref{Ass.uniform}) as we do in Theorem \ref{Thm.main0} above. Moreover, in addition to the smallness assumption (\ref{eq:4}), Kato does also require that $A$ and $B$ are $H_0$-smooth (which does not follow from (\ref{eq:4}) as we discussed in Remark \ref{Rem.smooth}) so the spectral stability results we obtain in Theorem \ref{Thm.main} and Corollary \ref{Cor1} are certainly not a consequence of  Kato's Theorem~\ref{Thm.main.similar}. 
Having made all these observations we of course also have to admit that in case that all of Kato's (and our) assumptions are satisfied, his conclusion that $H_V$ and $H_0$ are similar is considerably stronger than our observation that their spectra coincide. In particular, using Kato's result one can derive the following improved version of Corollary \ref{Cor2}.

\begin{Corollary}\label{Cor3}
Suppose Assumption~\ref{Ass.Ass}, \eqref{Ass.small1} and that $\Dom(A)=\Dom(B)=\Dom(|H_0|^{1/2})$. Moreover, suppose that $A$ and $B$ are closed and that $A=D_0B$ and $B=D_1A$ for some  $D_0,D_1 \in \mathscr{B}(\Hilbert')$. Then $H_V$ and $H_0$ are similar.
\end{Corollary}
\begin{proof}
  As the proof of Corollary \ref{Cor2} showed, given the above assumptions $A$ and $B$ are smooth relative to~$H_0$, hence Kato's theorem applies.
\end{proof}
\begin{Remark}
In particular, the previous corollary applies in case that $A=UB$, where $U \in \mathscr{B}(\Hilbert, \Hilbert')$ is a partial isometry with initial set $\overline{\Ran}(B)$, since then $B=U^*A$. This example is important in applications to Schr\"odinger operators,
see Section~\ref{Sec.app} below.
\end{Remark}

 
Having stated the advantages of Kato's and our own results, let us conclude this section by noting that Kato's proof of Theorem~\ref{Thm.main.similar} is very different from
our proof of Theorem~\ref{Thm.main}. In fact, he uses the method of stationary scattering theory
(and the similarity transformation $W$ he constructs has the meaning of a wave operator), while we work directly with the mentioned variants of the Birman--Schwinger principle.  

%
\subsection{Organisation of the paper}
In Section~\ref{Sec.pseudo} we introduce the operator~$H_V$
as the pseudo-Friedrichs extension of $H_0+V$. 
Sections~\ref{Sec.point}, \ref{Sec.res} and~\ref{Sec.ess}
are devoted to establishing the aforementioned variants 
of the Birman--Schwinger principle
for the point, residual and essential spectra, respectively.
In Section~\ref{Sec.proofs} we provide the proofs of
Theorem~\ref{Thm.main0} and Theorem~\ref{Thm.main}. 
Finally, in Section~\ref{Sec.app} we apply the abstract theorems
to Schr\"odinger and Dirac operators;
we recall some classical as well as recently established properties,
and prove completely new results for Schr\"odinger operators 
in three-dimensional hyperbolic space. Finally, the appendix contains a proof that Kato's extension $\tilde{H}_V$ and our pseudo-Friedrichs extension $H_V$ coincide given some suitable assumptions.

\section{The pseudo-Friedrichs extension}\label{Sec.pseudo}
%
By our standing Assumption~\ref{Ass.Ass},
$H_0$~is a self-adjoint operator in a complex separable Hilbert space~$\Hilbert$.
Recall (\cf~\cite[Sec.~VI.2.7]{Kato}) that
the absolute value $|H_0| := (H_0^2)^{1/2}$ is also self-adjoint,
$\Dom(|H_0|) = \Dom(H_0)$ is a core of $|H_0|^{1/2}$ 
and~$H_0$ and~$|H_0|$ commute
(in the sense of their resolvents).
The operator $G_0:\Dom(H_0)\to\Hilbert, G_0=|H_0|+1$ is bijective.
We define a sesquilinear form associated with~$H_0$ by
$$
  h_0(\phi,\psi) := \big(G_0^{1/2}\phi,H_0 G_0^{-1} G_0^{1/2}\psi\big) 
  \,, \qquad
  \phi,\psi \in \Dom(h_0) := \Dom(|H_0|^{1/2})
  \,.
$$
Since $H_0G_0^{-1} \in \mathscr{B}(\Hilbert)$ is selfadjoint, we see that $h_0$ is symmetric, \ie\
  $h_0(\phi,\psi)=\overline{h_0(\psi,\phi)}=:h_0^*(\phi,\psi)$ for
  $\phi,\psi \in \Dom(|H_0|^{1/2})$.  
Moreover, $h_0(\phi,\psi)=(\phi,H_0\psi)$ {and, by symmetry,
  $h_0(\psi,\phi)=(H_0\psi,\phi)$
for every $\phi \in \Dom(|H_0|^{1/2})$ and $\psi \in \Dom(H_0)$. 

Let $A: \Dom(A) \subset \Hilbert \to \Hilbert'$ 
and $B: \Dom(B) \subset \Hilbert \to \Hilbert'$ be two operators satisfying $\Dom(|H_0|^{1/2}) \subset \Dom(A) \cap \Dom(B)$ and (\ref{Ass.bounded}).
Our goal is to introduce a closed extension~$H_V$ 
of the operator sum $H_0+V$, with $V:=B^*A$,
as the \emph{pseudo-Friedrichs extension} \cite[Thm.~VI.3.11]{Kato}
(see also \cite{Veselic_2008} for more recent developments).
It is a suitable generalisation of the Friedrichs extension
in the case when~$H_0$ is not necessarily bounded from below. 

The idea is to replace~$V$ by its sesquilinear form 
$$
  v(\phi,\psi):=(B\phi,A\psi)
  \,, \qquad
  \phi,\psi \in \Dom(v) := \Dom(|H_0|^{1/2})
  \,.
$$  
Noting that, by assumption~(\ref{Ass.bounded}), 
we can rewrite $v$ as  
$$ 
  v(\phi,\psi) 
  = \big(B G_0^{-1/2} G_0^{1/2} \phi, 
  A G_0^{-1/2} G_0^{1/2} \psi\big) 
  = \big(G_0^{1/2}\phi, 
  [B G_0^{-1/2}]^*A G_0^{-1/2} G_0^{1/2} \psi\big), 
  \qquad \phi,\psi \in \Dom(|H_0|^{1/2}),
$$
we obtain that for 
\[ h_V:= h_0 +v, \quad \Dom(h_V):=\Dom(|H_0|^{1/2}),\] 
we have 
\begin{equation}
 h_V(\phi,\psi) = 
 \big( G_0^{1/2}\phi, 
 (H_0G_0^{-1}+ [B G_0^{-1/2}]^*A G_0^{-1/2}) G_0^{1/2} \psi\big), 
 \qquad \phi,\psi \in \Dom(|H_0|^{1/2}).\label{eq:8}
\end{equation}
Hence, we define
\begin{equation}\label{pseudo}
  H_V := G_0^{1/2} 
  \big(H_0 G_0^{-1}+[B G_0^{-1/2}]^*A G_0^{-1/2}\big) \, G_0^{1/2}
\end{equation}
with its natural domain. Clearly, $\Dom(H_V) \subset \Dom(|H_0|^{1/2})$ and if $\psi \in \Dom(H_0) \cap \Dom(V)$,
where $\Dom(H_0) \subset \Dom(|H_0|^{1/2})$ and
$
  \Dom(V) 
  = A^{-1} \Dom(B^*)
  = \{\psi\in\Dom(A):A\psi\in\Dom(B^*)\}
  \subset \Dom(|H_0|^{1/2})
$,
then for all $\phi$ in the dense set $\Dom(|H_0|^{1/2})$ we have
$
  (\phi,H_V\psi) 
  = (\phi,H_0\psi) + (\phi,V\psi)
$,
so $H_V\psi = (H_0+V)\psi$. This shows that $H_V \supset H_0+V$ and one has the  representation formula
\begin{equation}\label{form}
  \forall \phi \in \Dom(|H_0|^{1/2}), \ \psi \in \Dom(H_V)
  \,, \qquad
  (\phi,H_V\psi) 
  = h_V(\phi,\psi) 
  \,.
\end{equation}
Now let us verify that $H_V$ is a closed operator. We will do this by showing that $\rho(H_V)$ is non-empty. For this purpose, 
we use assumption~(\ref{Ass.spectrum}), \ie\ there exists $\lambda_0 \in \rho(H_0)$ such that $-1 \notin \sigma(K(\lambda_0))$. Using that $I \supset G_0^{1/2}G_0^{-1}G_0^{1/2}$, 
with this choice of  $\lambda_0$ we can write
\begin{align*}
  H_V - \lambda_0 = G_0^{1/2} 
  \big([H_0-\lambda_0] G_0^{-1}+[B G_0^{-1/2}]^*A G_0^{-1/2}\big)
  \, G_0^{1/2}.
\end{align*}
In particular, we obtain that
\begin{equation}
 (H_V - \lambda_0)^{-1}= G_0^{-1/2} 
 \big([H_0-\lambda_0] G_0^{-1}+[B G_0^{-1/2}]^*A G_0^{-1/2}\big)^{-1} G_0^{-1/2},\label{eq:6}
\end{equation}
provided that 
\[ [H_0-\lambda_0] G_0^{-1}+[B G_0^{-1/2}]^*A G_0^{-1/2} = [H_0-\lambda_0]G_0^{-1} \left( I + G_0(H_0-\lambda_0)^{-1} [BG_0^{-1/2}]^* A G_0^{-1/2} \right) \]
has a bounded inverse. 
But this is the case if, and only if,
$-1 \notin \sigma(G_0(H_0-\lambda_0)^{-1} [BG_0^{-1/2}]^* A G_0^{-1/2})$ which is the case (as we already argued in the proof of Lemma \ref{Lemma1} (ii)) if, and only if,  $-1 \notin \sigma(K(\lambda_0))$. So we conclude that indeed $\lambda_0 \in \rho(H_V)$ and $H_V$ is closed.

Next, let us show that 
$\Dom(H_V)=\Ran((H_V-\lambda_0)^{-1})$ is dense in $\Hilbert$. To this end, note that from (\ref{eq:6}) we obtain that 
\begin{equation}
 [(H_V - {\lambda}_0)^{-1}]^*= G_0^{-1/2} 
 \big([H_0-\overline{\lambda}_0] G_0^{-1}+[A G_0^{-1/2}]^*B G_0^{-1/2}\big)^{-1} G_0^{-1/2},\label{eq:7}
\end{equation}
where we used that $G_0^{-1/2} \in \mathscr{B}(\Hilbert)$ is self-adjoint and $[H_0-{\lambda}_0] G_0^{-1}+[B G_0^{-1/2}]^*A G_0^{-1/2}$ is invertible in $\mathscr{B}(\Hilbert)$. 
Now the operator on the right-hand side of (\ref{eq:7}) 
is clearly injective and hence  
$$
\overline{\Dom(H_V)}=\overline{\Ran(H_V-\lambda_0)^{-1}}
=\big(\Ker([(H_V-\lambda_0)^{-1}])^*\big)^\perp
=\Hilbert,
$$
so $H_V$ is densely defined. In particular, its adjoint $H_V^*$ exists and $\overline{\lambda}_0 \in \rho(H_V^*)$ with  
\begin{equation}
 (H_V^* - \overline{\lambda}_0)^{-1}
 = G_0^{-1/2} 
 \big([H_0-\overline{\lambda}_0] G_0^{-1}+[A G_0^{-1/2}]^*B G_0^{-1/2}\big)^{-1} G_0^{-1/2}.\label{eq:12}
\end{equation}
It follows that 
$\Dom(H_V^*)  \subset \Ran(G_0^{-1/2})= \Dom(|H_0|^{1/2})$ and
\begin{equation}\label{pseudo*}
  H_V^* = G_0^{1/2} 
  \big(H_0 G_0^{-1}+[A G_0^{-1/2}]^*B G_0^{-1/2}\big) \, G_0^{1/2}.
\end{equation}
Moreover, with the adjoint form 
$$
  v^*(\phi,\psi) 
  := \overline{v(\psi,\phi)} 
  = (A\phi,B\psi), \qquad
  \Dom(v^*)
  := \Dom(v)
  =\Dom(|H_0|^{1/2}),
$$
we obtain the representation formula
\begin{equation}\label{form*}
  \forall \phi \in \Dom(|H_0|^{1/2}), \ \psi \in \Dom(H_V^*)
  \,, \qquad
  (\phi,H_V^*\psi) = 
  h_V^*(\phi,\psi) 
  \,,
\end{equation}
where $h_V^* = h_0^*+v^*$.
	
Let us summarise the properties of the pseudo-Friedrichs extension
into the following theorem.

\begin{Theorem}\label{Thm.pseudo} 
Suppose Assumption~\ref{Ass.Ass} and set $V:=B^*A$.
There exists a unique closed extension~$H_V$ of $H_0 + V$  
such that $\Dom(H_V) \subset \Dom(|H_0|^{1/2})$,
$\Dom(H_V^*) \subset \Dom(|H_0|^{1/2})$ and the representation formulae~\eqref{form} and~\eqref{form*} hold. 
\end{Theorem}
\begin{proof}
It remains to verify the uniqueness claim.
Let~$\hat{H}_V$ be another closed extension of~$H_0+V$
with the properties stated in the theorem. 
Let $\phi \in \Dom(|H_0|^{1/2})$ 
and $\psi \in \Dom(\hat{H}_V) \subset \Dom(|H_0|^{1/2})$.
Then~\eqref{form} and (\ref{eq:8}) imply that
$$
  (\phi,\hat{H}_V\psi) = h_0(\phi,\psi) + v(\phi,\psi) 
  = \big(G_0^{1/2}\phi,H_0 G_0^{-1} G_0^{1/2}\psi\big)  
  + \big(G_0^{1/2}\phi,CG_0^{1/2}\psi\big)
  \,,
$$
where $C=[BG_0^{-1/2}]^*AG_0^{-1/2}$. But this implies that 
$[H_0G_0^{-1}+C]G_0^{1/2}\psi \in \Dom(G_0^{1/2})=\Dom((G_0^{1/2})^*)$  and 
$$
  \hat{H}_V\psi = G_0^{1/2} (H_0 G_0^{-1}+C) G_0^{1/2}\psi
  \,.
$$   
By~\eqref{pseudo}, it follows that $\psi \in \Dom(H_V)$
and $\hat{H}_V\psi=H_V\psi$. 
This shows that $\hat{H}_V \subset H_V$.

Now, let $\phi \in \Dom(|H_0|^{1/2})$ 
and $\psi \in \Dom((\hat{H}_V)^*) \subset \Dom(|H_0|^{1/2})$.
Then~\eqref{form*} implies 
$$
\begin{aligned}
  (\phi,(\hat{H}_V)^*\psi) 
  &= h_0^*(\phi,\psi) + v^*(\phi,\psi) = \overline{h_0(\psi,\phi)} + \overline{(B\psi, A\phi)} \\
  &= \big(H_0 G_0^{-1}G_0^{1/2}\phi,G_0^{1/2}\psi\big)  
  + \big(CG_0^{1/2}\phi,G_0^{1/2}\psi\big)
  \\
  &= \big(G_0^{1/2}\phi,H_0 G_0^{-1}G_0^{1/2}\psi\big)  
  + \big(G_0^{1/2}\phi,C^*G_0^{1/2}\psi\big)
  \,,
\end{aligned}
$$
where the second equality employs the commutativity of~$H_0$ and~$G_0$. Arguing as above, this implies that 
$$
  (\hat{H}_V)^*\psi = G_0^{1/2} (H_0 G_0^{-1}+C^*) G_0^{1/2}\psi
$$ 
and hence by \eqref{pseudo*} it follows that $\psi \in \Dom(H_V^*)$
and $(\hat{H}_V)^*\psi=H_V^*\psi$. This shows that $(\hat{H}_V)^* \subset H_V^*$,
so $\hat{H}_V \supset H_V$.
\end{proof}

We conclude this section about the pseudo-Friedrichs extension with the following generalised version of the second resolvent identity.
\begin{Proposition}\label{Prop.2nd}
For all $z \in \rho(H_0) \cap \rho(H_V)$,
\begin{equation}\label{2nd.pseudo}
  (H_V-z)^{-1} - (H_0-z)^{-1} 
  = - [B(H_0-\bar{z})^{-1}]^* A (H_V-z)^{-1} 
  \,.
\end{equation}
\end{Proposition}
\begin{proof}
Given any $f,g \in \Hilbert$, 
set $\phi:=(H_0-\bar{z})^{-1}f$ and $\psi:=(H_V-z)^{-1}g$.
Then
$$
\begin{aligned}
  \big(f,[(H_V-z)^{-1} - (H_0-z)^{-1}]g\big) 
  &= \big((H_0-\bar{z})\phi,\psi\big) -\big(\phi,(H_V-z)\psi\big) 
  \\
  &= (H_0\phi,\psi) - (\phi,H_V\psi)
  \\
  &= h_0(\phi,\psi) - h_V(\phi,\psi)
  \\
  &= (B\phi,A\psi)
  \\
  &= (B(H_0-\bar{z})^{-1}f,A(H_V-z)^{-1}g)
  \\
  &= (f,[B(H_0-\bar{z})^{-1}]^*A(H_V-z)^{-1}g)
  \,,
\end{aligned}
$$
where the third equality holds because both $\phi,\psi \in \Dom(|H_0|^{1/2})$.
\end{proof}
%

\section{The point spectrum}\label{Sec.point}
%
This section deals with the point spectrum of $H_V$. As a byproduct of the following two theorems, we obtain a proof of Theorem \ref{Thm.evs0}. For instance, the next theorem establishes the implication $\Longrightarrow$ 
of Theorem~\ref{Thm.evs0}. 
\begin{Theorem}\label{Thm.evs1} 
Suppose Assumption~\ref{Ass.Ass}.
Let $H_V\psi=\lambda\psi$ with some 
$\lambda \in \Com\setminus\sigma(H_0)$ 
and $\psi\in\Dom(H_V) \setminus \{0\}$. 
Then $g := A \psi \not= 0$ and $K_{\lambda} g = - g$.
\end{Theorem}
\begin{proof}
Suppose that $g=A\psi=0$. Then for every $f \in \Dom(H_0)$ we have
\[ (H_0f, \psi) = h_0(f,\psi) = h_V(f,\psi) - (Bf, A\psi) = h_V(f,\psi) = (f, H_V \psi) = (f, \lambda \psi).\]
This shows that $\psi \in \Dom(H_0^*)=\Dom(H_0)$ and $H_0 \psi =
H_0^*\psi=\lambda \psi$, so $\lambda \in \sigma_{\mathrm p}(H_0)$, a
contradiction. Hence $g \neq 0$.
 
Now for every $\phi \in \Hilbert$, one has
\begin{equation}\label{ev1}
\begin{aligned}
  (\phi,K_\lambda g)
  &= \big([AG_0^{-1/2}]^*\phi,
  [G_0(H_{0}-\lambda)^{-1}] [BG_0^{-1/2}]^*g\big)
  \\
  &= \big( [B G_0^{-1/2}]
  [G_0(H_{0}-\lambda)^{-1}]^* [AG_0^{-1/2}]^*  \phi,
  A\psi\big)
  \\
  &= (B\eta,A\psi)
  = v(\eta,\psi)
\end{aligned}
\end{equation}
with
$
  \eta := G_0^{-1/2}[G_0(H_{0}-\lambda)^{-1}]^* [AG_0^{-1/2}]^*\phi
  \in \Dom(|H_0|^{1/2})
  \,.
$
Using~\eqref{form}, it follows that
\begin{equation}\label{ev2}
\begin{aligned}
  (\phi,K_\lambda g) 
  &= (\eta,H_V\psi)	 
  - h_0(\eta,\psi) 
  \\
  &= \lambda \, (\eta,\psi) 
  - h_0(\eta,\psi) 
  \\
  &= \lambda \, \big(G_0^{1/2}\eta,G_0^{-1}G_0^{1/2}\psi\big)
  - \big(G_0^{1/2}\eta,H_0 G_0^{-1} G_0^{1/2}\psi\big)
  \\
  &= - \big(G_0^{1/2}\eta,(H_0-\lambda) G_0^{-1} G_0^{1/2}\psi\big)
  \\
  &= - \big([AG_0^{-1/2}]^*\phi,
  G_0(H_{0}-\lambda)^{-1}(H_0-\lambda) G_0^{-1} G_0^{1/2}\psi\big)
  \\
  &= - \big([AG_0^{-1/2}]^*\phi,
  G_0^{1/2}\psi\big)
  \\ 
  &= - \big(\phi,A\psi\big)
  = - (\phi,g)
  \,.
\end{aligned}
\end{equation}
Since this is true for every $\phi \in \Hilbert$,
it follows that $K_\lambda g=-g$.
\end{proof}

The following theorem establishes the opposite implication $\Longleftarrow$ 
of Theorem~\ref{Thm.evs0}. 
\begin{Theorem}\label{Thm.evs2} 
Suppose Assumption~\ref{Ass.Ass}.
Let $K_\lambda g=-g$ with some 
$\lambda \in \Com\setminus\sigma(H_0)$ 
and $g\in\Hilbert \setminus \{0\}$. 
Then $\psi := G_0^{1/2}(H_0-\lambda)^{-1} [BG_0^{-1/2}]^* g \in \Dom(H_V)$, 
$\psi\not=0$ and $H_V \psi = \lambda\psi$.
\end{Theorem}
\begin{proof}
Since $\psi \in \Dom(|H_0|^{1/2})$ we see that if $\psi=0$, then 
$
  0 = AG_0^{-1/2}G_0^{1/2}\psi = K_\lambda g = -g
$,
leading to a contradiction. Hence $\psi \neq 0$. Now for every $\phi \in \Dom(|H_0|^{1/2})$
$$
\begin{aligned}
  h_V(\phi,\psi) &= h_0(\phi,\psi) + v(\phi,\psi)= (G_0^{1/2}\phi, H_0G_0^{-1}G_0^{1/2} \psi) + (B\phi,A\psi) \\ 
  &= \big(G_0^{1/2}\phi,H_0(H_0-\lambda)^{-1}[BG_0^{-1/2}]^*g\big)
  + \big(B\phi,AG_0^{1/2}(H_0-\lambda)^{-1} [BG_0^{-1/2}]^* g\big)
  \\
  &= \big(G_0^{1/2}\phi,[BG_0^{-1/2}]^*g\big)
  + \lambda \big(G_0^{1/2}\phi,(H_0-\lambda)^{-1}[BG_0^{-1/2}]^*g\big)
  + (B\phi,K_\lambda g)
  \\
  &= \lambda \, (\phi,\psi).
\end{aligned}
$$
At the same time, by (\ref{eq:8})
$$ 
  h_V(\phi,\psi) = \big(G_0^{1/2}\phi, 
  (H_0G_0^{-1}+[BG_0^{-1/2}]^*AG_0^{-1/2}) G_0^{1/2}\psi\big)
$$ 
for every $\phi \in \Dom(|H_0|^{1/2})$. But this implies that 
$$ 
\big(H_0G_0^{-1}+[BG_0^{-1/2}]^*AG_0^{-1/2}\big) \, G_0^{1/2}\psi 
\in D(G_0^{1/2}),
$$
hence by \eqref{pseudo} we obtain that $\psi \in \Dom(H_V)$ and
$$ H_V \psi = G_0^{1/2} [H_0G_0^{-1}+[BG_0^{-1/2}]^*AG_0^{-1/2}] G_0^{1/2}\psi = \lambda \psi.$$
\end{proof}

We continue with a theorem extending the implication $\Longrightarrow$ 
of Theorem~\ref{Thm.evs0} to suitable points 
$\lambda \in \sigma(H_0)$.  
\begin{Theorem}\label{Thm.evs} 
Suppose Assumption~\ref{Ass.Ass}.
Let $H_V\psi=\lambda\psi$ with some $\lambda \in \sigma_{\mathrm{c}}(H_0)$ 
and $\psi\in\Dom(H_V) \setminus \{0\}$.
Then $g := A \psi \not=0$ and  
$
\displaystyle
  K_{\lambda + i\eps} \, g
  \xrightarrow[\eps \to 0^\pm]{w} - g
$. 
\end{Theorem}
\begin{proof}
As in the proof of Theorem~\ref{Thm.evs1} we see that $g \neq 0$.

Now we note that $\lambda$ is real and so $\lambda+i\eps \not\in \sigma(H_0)$
for all $\eps \in \Real \setminus \{0\}$. As in the proof of Theorem~\ref{Thm.evs1}, 
for every $\phi \in \Hilbert$, we have 
$$
\begin{aligned}
  (\phi,K_{\lambda+i\eps}g)
  &= - \big([AG_0^{-1/2}]^*\phi,
  G_0(H_{0}-\lambda-i\eps)^{-1}(H_0-\lambda) G_0^{-1} G_0^{1/2}\psi\big)
  \\
  &= - (\phi,g)
  - i I(\eps)
  \,,
\end{aligned}
$$
where
\begin{equation}\label{observation}
  I(\eps) := \eps \, \big([AG_0^{-1/2}]^*\phi,
  G_0(H_{0}-\lambda-i\eps)^{-1} G_0^{-1} G_0^{1/2}\psi\big)
  \,.
\end{equation}
It remains to show that $I(\eps)$ vanishes as $\eps \to 0$.
Using the spectral theorem, we have 
$$
\begin{aligned}
  I(\eps) &= \int_{\sigma(H_0)} 
  f(\eps)
  \ \der \big([AG_0^{-1/2}]^*\phi,E_0(r)G_0^{1/2}\psi\big)
\end{aligned} 
\qquad \mbox{with} \qquad
  f(\eps) := \frac{\eps}{r-\lambda-i\eps}
  \,,
$$
where~$E_0$ denotes the spectral measure of~$H_0$.
First, one has
$$
  f(\eps) \xrightarrow[\eps \to 0]{}
  \begin{cases}
    0 & \mbox{if} \quad r \not=\lambda \,,
    \\
    i & \mbox{if} \quad r =\lambda \,.
  \end{cases}
$$
In any case, however, $E_0(\{\lambda\})=0$
because $\lambda \not\in \sigma_\mathrm{p}(H_0)$.
Hence, $f(\eps) \to 0$ as $\eps \to 0$ 
almost everywhere with respect to the spectral measure. 
Second, neglecting the real part of $r-\lambda-i\eps$, 
one has 
$$
  |f(\eps)| \leq 
  \begin{cases}
    1 
    & \mbox{if} \quad \Im\lambda=0 \,,
    \\
    \displaystyle
    \frac{|\eps|}{|\Im\lambda+\eps|} \leq 1
    & \mbox{if} \quad \Im\lambda\not=0 \,,
  \end{cases}
$$
where the last inequality holds for all~$\eps$ 
with sufficiently small~$|\eps|$.
Hence $|f(\eps)|$ is bounded by an $\eps$-independent constant and 
$$
  \int_{\sigma(H_0)} 
  \der \big|
  \big([AG_0^{-1/2}]^*\varphi,E_0(r)G_0^{1/2}\psi\big)
  \big|
  \leq \big\|[AG_0^{-1/2}]^*\varphi\big\| \big\|G_0^{1/2}\psi\big\|
  < \infty
  \,.
$$
The dominated convergence theorem implies 
that $I(\eps) \to 0$ as $\eps \to 0$.
\end{proof}
\begin{Remark}
Until~\eqref{observation}, 
the proof of Theorem~\ref{Thm.evs} follows the lines of 
\cite[proof of Lem.~2]{FKV} or \cite[proof of Lem.~3]{FK9}
dealing with Schr\"odinger or Dirac operators, respectively.
To show that $I(\eps) \to 0$ as $\eps \to 0$ in the abstract case,
here we have developed a completely new approach. 
\end{Remark}
\begin{Corollary}\label{Corol.evs} 
Suppose Assumption~\ref{Ass.Ass}.
Let $\lambda \in \sigma_\mathrm{p}(H_V)$.
\begin{enumerate}
\item[\emph{(i)}]
If $\lambda \not\in \sigma(H_0)$, 
then $\|K_{\lambda}\| \geq 1$.
\item[\emph{(ii)}]
If $\lambda \in \sigma_\mathrm{c}(H_0)$, then
$
\displaystyle
  \liminf_{\eps\to 0^\pm} \|K_{\lambda+i\eps}\| \geq 1
$.
\end{enumerate}
\end{Corollary}
\begin{proof}
Let $\lambda \in \sigma_\mathrm{p}(H_V)$, 
let~$\psi \neq 0$ be a corresponding eigenvector
and set $\phi = A\psi \not= 0$.

If $\lambda \not\in \sigma(H_0)$, 
then Theorem~\ref{Thm.evs1} implies $\phi \neq 0$
$$
  \|\phi\|^2 \, \|K_\lambda\| 
  \geq |(\phi,K_\lambda\phi)| = \|\phi\|^2
  \,,
$$
from which the claim~(i) immediately follows.

If $\lambda \in \sigma_\mathrm{c}(H_0)$, 
we similarly write
$$
  \|\phi\|^2\, 
  \|K_{\lambda+i\eps}\|
  \geq |(\phi,K_{\lambda+i\eps}\phi)_\Hilbert| 
  \,.
$$
Taking the limit $\eps \to 0^\pm$,
Theorem~\ref{Thm.evs} implies
$$
  \|\phi\|^2	  \, 
  \liminf_{\eps\to 0^\pm} \|K_{\lambda+i\eps}\| 
  \geq \|\phi\|^2
  \,,
$$
from which the desired claim~(ii) immediately follows since, again, $\phi \neq 0$.
\end{proof}
%

\section{The residual spectrum}\label{Sec.res}
%
In view of the general characterisation 
(see, \eg, \cite[Prop.5.2.2]{KS-book}})
\begin{equation}\label{residual}
  \sigma_{\mathrm{r}}(H_V) 
  = \left\{\lambda \not\in \sigma_{\mathrm{p}}(H_V): \
  \bar{\lambda} \in \sigma_{\mathrm{p}}(H_V^*)\right\}
  ,
\end{equation}
the analysis of the residual spectrum of~$H_V$ can be reduced 
to the analysis of the point spectrum of the adjoint~$H_V^*$.

From the construction of the pseudo-Friedrichs extension
in Section~\ref{Sec.pseudo}, it is clear that the roles of~$A$ and~$B$
are just interchanged when considering~$H_V^*$.
It leads one to consider the adjoint Birman--Schwinger operator
\begin{equation}
  K_z^* =
  \big[BG_0^{-1/2}\big]
  \big[G_0(H_{0}-\bar{z})^{-1}\big] 
  \big[AG_0^{-1/2}\big]^*   
  \,.  
\end{equation}

In view of the above considerations,
Theorems~\ref{Thm.evs0}, \ref{Thm.evs1}, \ref{Thm.evs2}
and~\ref{Thm.evs} remain true if, in their statements, 
we simultaneously replace~$H_V$ by~$H_V^*$, $A$~by~$B$, $B$~by~$A$
and~$K_\lambda$ by~$K_{\bar\lambda}^*$ 
(notice the complex conjugate of~$\lambda$ in the latter).  
As a consequence of~\eqref{residual}, 
we therefore get the following theorem 
extending Theorem~\ref{Thm.evs0} to the residual spectrum.
\begin{Theorem}\label{Thm.res0} 
Suppose Assumption~\ref{Ass.Ass}. Then
\begin{equation*}
  \forall \lambda \in \Com \setminus \sigma(H_0) 
  \,, \qquad
  \lambda \in \sigma_{\mathrm{r}}(H_V) 
  \quad \Longleftrightarrow \quad
  -1 \in \sigma_{\mathrm{r}}(K_{\lambda}^*)
  \,.
\end{equation*}
\end{Theorem}

Similarly, we get the following theorem 
extending Theorem~\ref{Thm.evs} to the residual spectrum.
\begin{Theorem}\label{Thm.res} 
Suppose Assumption~\ref{Ass.Ass}.
Let $H_V^*\psi=\bar\lambda\psi$ with some 
$\lambda \in \sigma_{\mathrm{r}}(H_V) \cap \sigma_{\mathrm{c}}(H_0)$
and $\psi\in\Dom(H_V^*) \setminus \{0\}$.
Then $g := B \psi \not=0$ and  
$
\displaystyle
  K_{\lambda + i\eps}^* \, g
  \xrightarrow[\eps \to 0^\pm]{w} - g
$. 
\end{Theorem}

As consequence, we also get the following analogue of Corollary~\ref{Corol.evs}.  
\begin{Corollary}\label{Corol.res} 
Suppose Assumption~\ref{Ass.Ass}.
Let $\lambda \in \sigma_\mathrm{r}(H_V)$.
\begin{enumerate}
\item[\emph{(i)}]
If $\lambda \not\in \sigma(H_0)$, 
then $\|K_{\lambda}^*\| \geq 1$.
\item[\emph{(ii)}]
If $\lambda \in \sigma_\mathrm{c}(H_0)$, then
$
\displaystyle
  \liminf_{\eps\to 0^\pm} \|K_{\lambda+i\eps}^*\| \geq 1
$.
\end{enumerate}
\end{Corollary}
%

\section{The essential spectrum}\label{Sec.ess}
%
As mentioned in the introduction,
among the variety of definitions of essential spectra 
for non-self-adjoint operators, 
here we choose that of Wolf
(denoted by $\sigma_{\mathrm{e}_2}$ in~\cite[Chap.~IX.1]{Edmunds-Evans}).
That is, $\lambda \in \Com$ belongs 
to the \emph{essential spectrum} $\sigma_\mathrm{e}(H)$ 
of a closed operator~$H$ in~$\Hilbert$
if~$\lambda$ is an eigenvalue of infinite geometric multiplicity
or the range of $H-\lambda$ is not closed. 
This is equivalent to the existence of 
a sequence $\{\psi_n\}_{n \in \Nat} \subset \Dom(H)$
weakly convergent to zero
such that $\|\psi_n\|=1$ for every $n \in \Nat$
and $(H-\lambda)\psi_n \to 0$ as $n \to \infty$.

The following theorem is a modification of Theorem~\ref{Thm.evs}
to deal with the essential spectrum.
Note, however, that we do not require that 
the sequence is weakly converging to zero in this theorem.
The admissible points therefore satisfy 
$\lambda \in \sigma_\mathrm{e}(H_V) \cup \sigma_\mathrm{p}(H_V)$.
However, better results of Theorems~\ref{Thm.evs1} and~\ref{Thm.evs}
are available for eigenvalues.

\begin{Theorem}\label{Thm.cont} 
Suppose Assumption~\ref{Ass.Ass}.
Let $(H_V-\lambda)\psi_n \to 0$ as $n \to \infty$ 
with some $\lambda \in \Com\setminus\sigma(H_0)$ 
and $\{\psi_n\}_{n\in\Nat}\subset\Dom(H_V)$
such that $\|\psi_n\|=1$ for all $n \in \Nat$.
Then $\phi_n := A\psi_n \not=0$ 
for all sufficiently large~$n$ and
\begin{equation}\label{BS.cont}
  \lim_{n \to \infty} 
  \frac{(\phi_n,K_{\lambda} \phi_n)}{\|\phi_n\|^2} = - 1 
  \,.
\end{equation}
\end{Theorem}
\begin{proof}
First of all, let us show that $\phi_n \not= 0$
for all sufficiently large~$n$. 
In fact, we establish the stronger fact that
\begin{equation}\label{liminf}
  \liminf_{n \to \infty} \|\phi_n\| > 0
  \,.
\end{equation}
By contradiction, let us assume that
there exists a subsequence 
$\{\phi_{n_j}\}_{j\in\Nat}$
such that $n_j \to \infty$ and $\phi_{n_j}=A\psi_{n_j} \to 0$ as $j \to \infty$. 
From the identity~\eqref{form} and the hypothesis, we deduce that for $f_j:=(H_0-\overline{\lambda})^{-1}\psi_{n_j}$ we have
\begin{align*}
|h_0(f_j,\psi_{n_j}) - \lambda \, (f_j,\psi_{n_j})| &\leq |h_V(f_j,\psi_{n_j})-\lambda \, (f_j,\psi_{n_j})| +|(Bf_j,A\psi_{n_j})| \\ 
&= |(f_j,(H_V-\lambda)\psi_{n_j})| + |(B(H_0-\overline{\lambda})^{-1}\psi_{n_j},\phi_{n_j})| \\
&\leq \|f_j\| \|(H_V-\lambda)\psi_{n_j}\| + \|B(H_0-\overline{\lambda})^{-1}\| \|\phi_{n_j}\| \\
&\leq \|(H_0-\overline{\lambda})^{-1}\| \|(H_V-\lambda)\psi_{n_j}\| + \|B(H_0-\overline{\lambda})^{-1}\| \|\phi_{n_j}\|.
\end{align*} 
Here we used that $B(H_0-\overline{\lambda})^{-1}= ( B G_0^{-1/2}) (
G_0^{1/2} (H_0-\overline{\lambda})^{-1}) \in \mathscr{B}(\Hilbert,
\Hilbert')$. In particular, we see that  $|h_0(f_j,\psi_{n_j}) -
\lambda \, (f_j,\psi_{n_j})| \to 0$ for $j \to \infty$. On the other
hand, since $f_j \in \Dom(H_0)$ we also have   
\begin{align*}
 h_0(f_j,\psi_{n_j}) - \lambda \, (f_j,\psi_{n_j}) 
  = \big((H_0-\overline \lambda) f_j, \psi_{n_j}\big) 
  = \|\psi_{n_j}\|^2 = 1
\end{align*} 
for every $j \in \Nat$, which leads to a contradiction. Hence 
$\phi_n \not= 0$ for all sufficiently large~$n$ and~\eqref{liminf} holds true.

The rest of the proof is similar to that of Theorem~\ref{Thm.evs}.
Since $\lambda \not\in \sigma(H_0)$,
\eqref{ev1}~implies
%
$
  (\phi_n,K_{\lambda} \phi_n)
  = v(\eta_n,\psi_n) 
$,
%
where  
$
  \eta_n := G_0^{-1/2}[G_0(H_{0}-\lambda)^{-1}]^* [AG_0^{-1/2}]^*\phi_n
$
belongs to $\Dom(|H_0|^{1/2})$ and $\|\eta_n\| \leq C_0 \|\phi_n\|$
with some constant~$C_0$ independent of~$n$.
In analogy with~\eqref{ev2}, we have
\begin{equation*} 
\begin{aligned}
  v(\eta_n,\psi_n)
  &= h_V(\eta_n,\psi_n)	 
  - h_0(\eta_n,\psi_n)
  \\
  &= \big(\eta_n,(H_V-\lambda)\psi_n\big)	 
  + \lambda \, (\eta_n,\psi_n) 
  - h_0(\eta_n,\psi_n) 
  \\
  &= \big(\eta_n,(H_V-\lambda)\psi_n\big) - \|\phi_n\|^2
  \,.
\end{aligned}
\end{equation*}
Consequently,
$$
  \left|
  \frac{(\phi_n,K_{\lambda} \phi_n)}{\|\phi_n\|^2} + 1
  \right|
  = \frac{\big|\big(\eta_n,(H_V-\lambda)\psi_n\big)\big|}{\|\phi_n\|^2}
  \leq \frac{\|\eta_n\|}{\|\phi_n\|^2} \, \|(H_V-\lambda)\psi_n\|
  \leq \frac{C_0}{\|\phi_n\|} \, \|(H_V-\lambda)\psi_n\|
  \,.
$$ 
Using~\eqref{liminf} and the hypothesis, we get the desired claim.
\end{proof}
\begin{Remark}
Theorem~\ref{Thm.cont} is inspired by \cite[Lem.~3]{FKV}
proved for Schr\"odinger operators with $\lambda \in \Com\setminus\Real$.
Here we have developed an abstract approach
and included real points $\lambda \not\in \sigma(H_0)$ as well. 
\end{Remark}
\begin{Corollary}\label{Corol.cont} 
Suppose Assumption~\ref{Ass.Ass}.
If 
$
  \lambda \in 
  [\sigma_\mathrm{e}(H_V) \cup \sigma_\mathrm{p}(H_V)]
  \setminus \sigma(H_0)
$,
then $\|K_{\lambda}\| \geq 1$.
\end{Corollary}
\begin{proof}
Let $\lambda \not \in \sigma(H_0)$.
If $\lambda \in \sigma_\mathrm{p}(H_V)$,
then the claim follows from part~(i) of Corollary~\ref{Corol.evs}.
However, the following alternative argument applies as well.
Given any $\lambda \in \sigma_\mathrm{e}(H_V) \cup \sigma_\mathrm{p}(H_V)$, 
let $\{\psi_n\}_{n \in \Nat} \subset \Dom(H_V)$ be a corresponding sequence
satisfying $\|\psi_n\|=1$ for every $n \in \Nat$
and $H_V\psi_n-\lambda\psi_n \to 0$ as $n \to \infty$.
By Theorem~\ref{Thm.cont},
the sequence $\{\psi_n\}_{n \in \Nat}$ defined by $\phi_n = A\psi_n$
has non-zero elements for all sufficiently large~$n$ and
\begin{equation*} 
  \|K_{\lambda}\|
  \geq 
  \lim_{n \to \infty} 
  \frac{|(\phi_n,K_{\lambda} \phi_n)|}{\|\phi_n\|^2} =  1 
  \,,
\end{equation*}
where the estimate is due to the Schwarz inequality.
\end{proof}
%

\section{The remaining proofs}\label{Sec.proofs}

\subsection*{Proof of Theorem \ref{Thm.main}}
First, let us note that given (\ref{Ass.small1}),  Corollary~\ref{Corol.evs} implies that
$ \sigma_\mathrm{p}(H_V) \subset \sigma_\mathrm{p}(H_0)$ and, 
noting that  $\|K_z\|=\|K_z^*\|$ for every $z \in \rho(H_0)$, Corollary~\ref{Corol.res} implies that 
$ \sigma_\mathrm{r}(H_V) \subset \sigma_\mathrm{p}(H_0)$. 
Here we used that the residual spectrum of a self-adjoint operator is empty. Taken together we thus showed that
\begin{equation}
 [\sigma_\mathrm{p}(H_V) \cup \sigma_\mathrm{r}(H_V)] 
 \subset \sigma_\mathrm{p}(H_0),\label{eq:9}
\end{equation}
which is the first statement of part (ii) of Theorem \ref{Thm.main}. Now let us note that in general, 
$\sigma_\mathrm{c}(H_V) \subset \sigma_\mathrm{e}(H_V)$, so by  Corollary~\ref{Corol.cont} we obtain that
\begin{equation}
  \label{eq:10}
  \sigma_\mathrm{c}(H_V) \subset \sigma(H_0).
\end{equation}
The inclusions (\ref{eq:9}) and (\ref{eq:10}) ensure that $\sigma(H_V) \subset \sigma(H_0)$. Since the reverse inclusion will be shown in the proof of Theorem~\ref{Thm.main0} (which, to be sure, does not rely in any way on the results of Theorem~\ref{Thm.main}), we obtain that $\sigma(H_V)=\sigma(H_0)$, which is part (i) of Theorem \ref{Thm.main}. In particular, this implies that $\sigma_\mathrm{c}(H_0) \subset \sigma(H_V)$ and since $\sigma_\mathrm{p}(H_0) \cap \sigma_\mathrm{c}(H_0) = \emptyset$, the inclusion (\ref{eq:9}) implies that $\sigma_\mathrm{c}(H_0) \subset \sigma_\mathrm{c}(H_V)$, which is the second statement of part (ii) of Theorem \ref{Thm.main}. This concludes the proof.\qed

\subsection*{Proof of Theorem \ref{Thm.main0}}

Interestingly enough, and in contrast to the proof of Theorem \ref{Thm.main} given above, the proof of this theorem does not rely on the Birman--Schwinger principles. We will prove the theorem by contradiction.

So assume that (\ref{Ass.uniform}) holds and set $C_0:= \sup_{z \in \Com \setminus \Real} \|K_z\|<\infty$. Moreover, let us  suppose that there exists $\lambda_0 \in \sigma(H_0) \cap \rho(H_V)$. We will derive a contradiction in four steps:

\medskip
\textbf{Step 1}. 
Since $\lambda_0 \in \sigma(H_0)$ and $H_0$ is selfadjoint there exists a sequence 
$\{f_n\}$ in $\Dom(H_0)$ 
such that $\|f_n\|=1, n \in \Nat,$ and $(H_0-\lambda_0)f_n \to 0$ for
$n \to \infty$. In particular, since $\lambda_0 \in \Real$ for every
$\lambda \in \Com \setminus \Real$ we obtain that 
\begin{equation}
  \label{eq:11}
  A[(H_0-\lambda)^{-1}-(\lambda_0-\lambda)^{-1}]f_n = (\lambda_0-\lambda)^{-1}A(H_0-\lambda)^{-1}(\lambda_0-H_0)f_n \to 0 \qquad (n \to \infty).
\end{equation}
Here we used that $A(H_0-\lambda)^{-1}=(AG_0^{-1/2})(G_0^{1/2}(H_0-\lambda)^{-1}) \in \mathscr{B}(\Hilbert,\Hilbert')$ by assumption (\ref{Ass.bounded}).

\medskip
\textbf{Step 2} (compare the proof of Theorem \ref{Thm.cont}). We have $L:= \liminf_{n \to \infty} \|Af_n\| > 0$. Indeed, suppose that this is not the case. Then there would exist a subsequence 
$\{f_{n_j}\}$ of $\{f_n\}$ 
such that $Af_{n_j} \to 0$ for $j \to \infty$. Since $\lambda_0=\overline{\lambda}_0 \in \rho(H_V^*)$, we could then estimate 
\begin{align*}
  &  |h_V^*(f_{n_j},(H_V^*-\lambda_0)^{-1}f_{n_j}) - \lambda_0 (f_{n_j},(H_V^*-\lambda_0)^{-1}f_{n_j})|\\
   &= |h_0^*(f_{n_j},(H_V^*-\lambda_0)^{-1}f_{n_j}) + v^*(f_{n_j},(H_V^*-\lambda_0)^{-1}f_{n_j})- \lambda_0(f_{n_j},(H_V^*-\lambda_0)^{-1}f_{n_j})| \\
  &= |\overline{h_0((H_V^*-\lambda_0)^{-1}f_{n_j},f_{n_j})} + (Af_{n_j},B(H_V^*-\lambda_0)^{-1}f_{n_j})- \lambda_0(f_{n_j},(H_V^*-\lambda_0)^{-1}f_{n_j})|\\
  &=    |((H_0-\lambda_0)f_{n_j},(H_V^*-\lambda_0)^{-1}f_{n_j}) + (Af_{n_j},B(H_V^*-\lambda_0)^{-1}f_{n_j})| \\
  &\leq \|(H_0-\lambda_0)f_{n_j}\| \|(H_V^*-\lambda_0)^{-1}\| + \|Af_{n_j}\| \|B(H_V^*-\lambda_0)^{-1}\|.
\end{align*}
Here we used that $B(H_V^*-\lambda_0)^{-1} \in \mathscr{B}(\Hilbert,\Hilbert')$ as follows from (\ref{eq:12}) 
and assumption~\eqref{Ass.bounded}. 
In particular, we see that
\[ h_V^*(f_{n_j},(H_V^*-\lambda_0)^{-1}f_{n_j}) - \lambda_0 (f_{n_j},(H_V^*-\lambda_0)^{-1}f_{n_j}) \to 0 \qquad (j \to \infty).\]
On the other hand, since $(H_V^*-\lambda_0)^{-1}f_{n_j}  \in D(H_V^*)$, we also obtain from (\ref{form*}) that
\[ h_V^*(f_{n_j},(H_V^*-\lambda_0)^{-1}f_{n_j}) - \lambda_0 (f_{n_j},(H_V^*-\lambda_0)^{-1}f_{n_j}) = (f_{n_j},(H_V^*-\lambda_0)(H_V^*-\lambda_0)^{-1}f_{n_j}) = \|f_{n_j}\|^2=1\]
for all $j \in \Nat$, which leads to a contradiction. Hence $L=\liminf_{n \to \infty} \|Af_n\| > 0$.

\medskip
\textbf{Step 3}. 
Now let $\eps_0 > 0$ such that $\lambda_0 + i\eps \in \rho(H_0) \cap \rho(H_V)$ for all $\eps \in (0,\eps_0)$. Then using the resolvent identity (\ref{2nd.pseudo}), the triangle inequality and the fact that
$$A[B(H_0-\lambda_0+i\eps)^{-1}]^*= A[BG_0^{-1/2}G_0(H_0-\lambda_0+i\eps)^{-1}G_0^{-1/2}]^* =K(\lambda_0+i\eps) ,$$ for all $\eps \in (0,\eps_0)$ we obtain that
\begin{align*}
   \|A(H_V-\lambda_0-i\eps)^{-1}\| &= \|A(H_0-\lambda_0-i\eps)^{-1} -A[B(H_0-\lambda_0+i\eps)^{-1}]^*A(H_V-\lambda_0-i\eps)^{-1}\| \\
                                  &\geq \|A(H_0-\lambda_0-i\eps)^{-1}\| - \|A[B(H_0-\lambda_0+i\eps)^{-1}]^*A(H_V-\lambda_0-i\eps)^{-1}\| \\
                                 &\geq \|A(H_0-\lambda_0-i\eps)^{-1}\| - C_0\|A(H_V-\lambda_0-i\eps)^{-1}\|.
\end{align*}
Hence for all $\eps \in (0,\eps_0)$ and $n \in \Nat$ we obtain (with the $f_n$'s as in Step 1) that
\begin{align}
  \|A(H_V-\lambda_0-i\eps)^{-1}\|  &\geq (1+C_0)^{-1}\|A(H_0-\lambda_0-i\eps)^{-1}\| \nonumber \\
  &\geq (1+C_0)^{-1}\|A(H_0-\lambda_0-i\eps)^{-1}f_n\|.  \label{eq:14}
\end{align}

\medskip
\textbf{Step 4}. 
Now fix some $\eps \in (0,\eps_0)$ and choose $n(\lambda_0,\eps)\in \Nat$ such that, using (\ref{eq:11}) with $\lambda=\lambda_0+i\eps$, we have
\[ \| A[(H_0-\lambda_0-i\eps)^{-1}-(-i\eps)^{-1}]f_n\| \leq 1 \qquad
  (n \geq n(\lambda_0,\eps)).\] 
The triangle inequality implies that for $n \geq n(\lambda_0,\eps)$
 \[ \|A[(H_0-\lambda_0-i\eps)^{-1}f_n\| \geq \frac 1 \eps \|Af_n\| -1\]
and hence using (\ref{eq:14}) we obtain that
\begin{equation}
  \label{eq:15}
  \| A(H_V-\lambda_0-i\eps)^{-1}\| \geq (1+C_0)^{-1} \left( \frac 1 \eps \|Af_n\| -1 \right), \qquad n \geq n(\eps,\lambda_0).
\end{equation}
Now consider the limes inferior of both sides of (\ref{eq:15}) with respect to $n \to \infty$ and use Step 2 to obtain that 
\[ \| A(H_V-\lambda_0-i\eps)^{-1}\| \geq (1+C_0)^{-1} \left( \frac  L \eps - 1 \right).\]
But since $L>0$ and $\eps \in (0,\eps_0)$ was arbitrary, this implies that
\begin{equation}
\limsup_{\eps \to 0} \|A(H_V-\lambda_0-i \eps)^{-1}\| = \infty.\label{eq:16}
\end{equation}
But $\lambda_0 \in \rho(H_V)$ and the function
$$\lambda \mapsto A(H_V-\lambda)^{-1}=A(H_V-\lambda_0)^{-1}+(\lambda-\lambda_0)A(H_V-\lambda_0)^{-1}(H_V-\lambda)^{-1}$$
is analytic (hence continuous) in a neighbourhood of $\lambda_0$, so
$$\lim_{\eps \to 0}\|A(H_V-\lambda_0-i \eps)^{-1}\|  = \|A(H_V-\lambda_0)^{-1}\|< \infty$$
(that $A(H_V-\lambda_0)^{-1}\in \mathscr{B}(\Hilbert,\Hilbert')$ can be seen by writing the operator as $[AG_0^{-1/2}][G_0^{1/2}(H_V-\lambda_0)^{-1}]$, which is okay since 
$\Dom(H_V) \subset \Dom(|H_0|^{1/2})$, and noting that here the first operator is bounded by (\ref{Ass.bounded}) and the second is bounded by the closed graph theorem). This contradicts (\ref{eq:16}) and hence $\sigma(H_0) \cap \rho(H_V)$ must be empty, 
\ie\ $\sigma(H_0) \subset \sigma(H_V)$.
\qed 

\section{Applications}\label{Sec.app}
%
In this section, we apply the abstract theorems to concrete problems.

\subsection{Schr\"odinger operators in the Euclidean spaces}
Given any positive integer~$d$,
let $H_0 := -\Delta$ in $\Hilbert:=\sii(\Real^d)$
with $\Dom(H_0) := H^2(\Real^d)$.
One has $\sigma(H_0)=[0,+\infty)$
and the spectrum is purely absolutely continuous.
The absolute value~$|H_0|$ satisfies 
$\||H_0|^{1/2}\psi\|=\|\nabla\psi\|$
for every $\psi \in \Dom(|H_0|^{1/2})=H^1(\Real^d)$.

Given any $V \in L_\mathrm{loc}^1(\Real^d)$,
we use the decomposition 
\begin{equation}\label{decomposition} 
  V(x) = \sgn V(x) \, |V(x)|
  = \sgn V(x) \, |V(x)|^{1/2} \, |V(x)|^{1/2} 
\end{equation}
for almost every $x \in \Real^d$.
Here $\sgn z := z/|z|$ if $z \not= 0$ and $\sgn z:=0$ if $z=0$.
We choose $A(x):=|V(x)|^{1/2}$ 
and $B(x) := \sgn\overline{V}(x) \, |V(x)|^{1/2}$.
We use the same symbols $A,B$ 
for the associated operators of multiplication 
with $\Dom(A) = \Dom(B) = \Dom(|H_0|^{1/2})$. 
Note that by the Sobolev inequality, 
a sufficient condition to satisfy~\eqref{eq:2} is 
$V = V_1 + V_2$ 
with $V_1 \in L^{p}(\Real^d)$ and $V_2 \in L^\infty(\Real^d)$,
where
\begin{equation}\label{Sobolev}
  \begin{aligned}
    &p=1 & \mbox{if} \quad d = 1 \,,
    \\
    &p>1 & \mbox{if} \quad d = 2 \,,
    \\
    &p=d/2 & \mbox{if} \quad d \geq 3 \,,
  \end{aligned}
  \qquad \mbox{and}\qquad
  \|V_1\|_{L^{p}(\Real^d)} < C_{p,d} 
  \,. 
\end{equation}
Here $C_{1,1} := \infty$ 
(the largeness of the norm $\|V_1\|_{L^{1}(\Real)}$ 
is unrestricted if $d=1$)
and $C_{p,d} := d(d-2)|\Sphere^d|/4$ if $d \geq 3$,
where $|\Sphere^d|$ denotes the volume of the $d$-dimensional unit sphere
(\cf~\cite[Thm.~8.3]{LL}). 
If $d=2$, an estimate on the constant~$C_{p,2}$
is also known (\cf~\cite[Thm.~8.5(ii)]{LL}),
but we shall not need it. 
In summary, $V$~falls within the class
of perturbations considered in Assumption~\ref{Ass.Ass}
and the pseudo-Friedrichs extension~$H_V$ is well defined. 

\begin{Remark}\label{Rem.Friedrichs}
Since~$H_0$ is bounded from below, the associated form 
$h_0[\psi]= \||H_0|^{1/2}\psi\|^2 = \|\nabla\psi\|^2$,
$\Dom(h_0) = H^1(\Real^d)$,  
is closed and bounded from below.
The form of the perturbation~$V$ reads 
$v[\psi]=\int_{\Real^d}V|\psi|^2$, $\Dom(v)=H^1(\Real^d)$.
Under our assumption~\eqref{Sobolev}, 
the perturbed form~$h_V$ is closed and sectorial with
$\Dom(h_V)=\Dom(h_0) = H^1(\Real^d)$.
Since the Friedrichs extension of the operator $H_0+V$ 
initially defined on $\mathcal{D}:=C_0^\infty(\Real^d)$ 
is the only m-sectorial
extension with domain contained in $\Dom(h_V)$
(\cf~\cite[Thm.~VI.2.11]{Kato}),
it follows that the pseudo-Friedrichs extension~$H_V$ 
defined by Theorem~\ref{Thm.pseudo} 
is actually the usual Friedrichs extension.
\end{Remark}

Spectral properties of~$H_V$ substantially differ
in \emph{high} dimensions $d \geq 3$
and \emph{low} dimensions $d = 1,2$.

\subsubsection{High dimensions}
Applying the abstract results of 
Theorems~\ref{Thm.main} and~\ref{Thm.main.similar},
we get the following spectral stability.
\begin{Theorem}[{\cite[Thm.~6.4]{Kato_1966}}
{\cite[Thm.~2]{Frank_2011}} \& {\cite[Thm.~3.2]{Frank-Simon_2017}}]%
\label{Thm.stability}
Let $d \geq 3$ and $V \in L^{d/2}(\Real^d)$.
There exists a positive dimensional constant~$c_d$ such that if
\begin{equation}\label{Frank}
  \|V\|_{L^{d/2}(\Real^d)} < c_d
  \,,
\end{equation}
then 
$$
  \sigma(H_V) = \sigma_\mathrm{c}(H_V) = [0,+\infty)
  \,.
$$
Moreover, $H_V$ and~$H_0$ are similar to each other.
\end{Theorem}
\begin{proof}
The idea of the proof in all dimensions $d\geq 3$ is due to Frank~\cite{Frank_2011}.
Based on a uniform Sobolev inequality due to~\cite{Kenig-Ruiz-Sogge_1987},
Frank established the resolvent estimate
(\cf~\cite[Eq.~(8)]{Frank_2011})
$$
  \forall z \in \Com\setminus[0,+\infty) 
  \,, \qquad
  \|(H_0-z)^{-1}\|_{L^p(\Real^d) \to L^{p'}(\Real^d)}
  \leq k_{p,d} \, |z|^{-(d+2)/2+d/p} 
  \,,
$$
where $2d/(d+2) \leq p \leq 2(d+1)/(d+3)$, $1/p+1/p'=1$
and~$k_{p,d}$ is a positive constant. 
For every $\phi,\psi \in H^1(\Real^d)$
and $z \not\in [0,+\infty)$ 
we obtain, taking Remark~\ref{rem1} into account,   
$$
\begin{aligned}
  |(\phi,K_z\psi)| 
  &= \big|\big(|V|^{1/2}\phi,(H_0-z)^{-1}|V|^{1/2}\tilde{\psi}\big)\big|
  \\
  &\leq 
  k_{p,d} \, |z|^{-(d+2)/2+d/p} \,
  \||V|^{1/2}\phi\|_{L^p(\Real^d)}
  \||V|^{1/2}\tilde{\psi}\|_{L^p(\Real^d)}
  \\
  &\leq 
  k_{p,d} \, |z|^{-(d+2)/2+d/p} \,
  \|V\|_{L^{p/(2-p)}(\Real^d)}
  \|\phi\| \|\tilde\psi\|
  \,, 
\end{aligned}
$$ 
where $\tilde\psi := (\sgn\bar{V})\psi$, 
so $\|\tilde\psi\|= \|\psi\|$.
Since~$H^1(\Real^d)$ is dense in $\sii(\Real^d)$,
this inequality extends to the whole Hilbert space
and we get 
$$
  \|K_z\| \leq k_{p,d} \, |z|^{-(d+2)/2+d/p} \,
  \|V\|_{L^{p/(2-p)}(\Real^d)}
  \,.
$$
Choosing $p:=2d/(d+2)$, 
we get the uniform (\ie~$z$-independent) bound
$$
  \|K_z\| \leq k_{p,d} \, 
  \|V\|_{L^{d/2}(\Real^d)}
  \,.
$$
By assuming~\eqref{Frank} with $c_d := k_{p,d}^{-1}$, 
we get the validity of~\eqref{Ass.small1}.
It follows by Theorem~\ref{Thm.main} that 
the spectrum of~$H_V$ is purely continuous
and equal to $[0,+\infty)$.
Furthermore, the same estimates as above ensure 
that the supremum in~\eqref{eq:5} is bounded
(also for~$A$ being replaced by~$B$)
from above by $2 k_{p,d} \, \|V\|_{L^{d/2}(\Real^d)}$.
Consequently, $A,B$ are $H_0$-smooth 
and hence similarity of $H_0$ and $H_V$ follows by
Corollary \ref{Cor3}.  
\end{proof}
\begin{Remark}
Assuming smallness of~$V$ in \emph{different} scales of Lebesgue spaces,
Theorem~\ref{Thm.stability} comes back to Kato \cite[Thm.~6.4]{Kato_1966}.
The identification of the optimal Lebesque space $L^{d/2}(\Real^d)$
(thanks to the availability of the uniform Sobolev inequality
\cite{Kenig-Ruiz-Sogge_1987})
and the present proof
is due to Frank~\cite[Thm.~2]{Frank_2011},
who established the absence of (discrete) eigenvalues of~$H_V$ outside $[0,+\infty)$.
In \cite[Thm.~3.2]{Frank-Simon_2017}, Frank and Simon excluded
(embedded) eigenvalues inside $[0,+\infty)$ as well.
The novelty of our statement here is that we additionally show that 
Frank's condition actually implies the stability of the continuous 
and residual spectra, too,
and even the similarity of~$H_V$ to~$H_0$.
\end{Remark}

For physical applications in dimension $d=3$, 
the space $L^{3/2}(\Real^3)$ is too restrictive, 
for it excludes potentials with critical singularities $V(x) \sim |x|^{-2}$
as $x \to 0$.
To include the singular potentials,
Frank~\cite[Thm.~3]{Frank_2011} showed that the $L^{3/2}$-norm
can be replaced by the Morrey--Campanato norm.
Alternatively, one can use the following old observation of Kato. 
\begin{Theorem}[{\cite[Thm.~6.1]{Kato_1966}}]%
\label{Thm.stability.3D.Kato}
Let $d = 3$ and $V \in L_\mathrm{loc}^{1}(\Real^d)$.
Let~$L$ be the integral operator in $\sii(\Real^3)$ with the kernel 
$$
  \frac{|V(x)|^{1/2} \, |V(y)|^{1/2}}{4\pi \, |x-y|} 
  \,.
$$
If~$L$ is bounded and there exists a constant $c<1$ such that
\begin{equation}\label{Kato}
  \|L\| \leq c 
  \,,
\end{equation}
then the conclusions of Theorem~\ref{Thm.stability} hold true.
\end{Theorem}
\begin{proof}
The idea of the proof is based on the explicit knowledge 
of the integral kernel of $(H_0-z)^{-1}$ in~$\Real^3$:
\begin{equation}\label{Green}
  G_z(x,y) := \frac{e^{-\sqrt{-z}\,|x-y|}}{4\pi \, |x-y|}
  \,,
\end{equation}
where $z \in \Com\setminus(0,+\infty)$
and $x,y \in \Real^3$ with $x \not= y$.
We use the branch of the square root on $\Com \setminus (-\infty,0]$ 
with positive real part.
The peculiarity of dimension $d=3$ is that 
one has the uniform pointwise bound
\begin{equation}\label{peculiar}
  \forall z \in \Com\setminus(0,+\infty)
  \,, \
  x,y \in \Real^3, \, x \not= y
  \,, \qquad
  |G_z(x,y)| \leq G_0(x,y)
  \,.
\end{equation}
Consequently, 
for every $\phi,\psi \in C_0^\infty(\Real^3)$, one has
%
\begin{align}\label{Kolja}
  |(\phi,K_z\psi)|
  &\leq 
  \int_{\Real^3}\int_{\Real^3} 
  (|V|^{1/2}|\phi|)(x) \, |G_z(x,y)| \, (|V|^{1/2}|\psi|)(y) 
  \, \der x \, \der y
  \nonumber \\
  &\leq  
  \int_{\Real^3}\int_{\Real^3} 
  (|V|^{1/2}|\phi|)(x) \, G_0(x,y) \, (|V|^{1/2}|\psi|)(y) 
  \, \der x \, \der y
  \,.
\end{align}
%
Note that the last integral is well defined because 
the functions $\phi,\psi$ are assumed to have a compact support. 
Using the definition of~$L$ and the fact that the space 
$C_0^\infty(\Real^3)$ is dense in $\sii(\Real^3)$,
one gets 
$$
  |(\phi,K_z\psi)| \leq (|\phi|,L\,|\psi|)
  \leq c \|\phi\| \|\psi\|
$$
for every $\phi,\psi \in \sii(\Real^3)$.
Consequently, $\|K_z\| \leq c$ uniformly in $z \in \Com\setminus[0,+\infty)$,
so~\eqref{Ass.small1} holds true.
Furthermore, the same estimates as above ensure 
that the supremum in~\eqref{eq:5}
is bounded from above by~$2c$.
Hence, the sufficient conditions of the abstract 
Theorem~\ref{Thm.main} and Corollary \ref{Cor3} are satisfied.
\end{proof} 

It is desirable to obtain sufficient conditions which guarantee
the validity of~\eqref{Kato}. 
An obvious choice is to bound the operator norm of~$L$
by its Hilbert--Schmidt norm leading to the sufficient condition
\begin{equation}\label{Rollnik}
  \|V\|_{R(\Real^3)} 
  := 
  \sqrt{
  \int_{\Real^3}\int_{\Real^3} 
  \frac{|V(x)| \, |V(y)|}{|x-y|^2} 
  \, \der x \, \der y
  }
  < 4\pi
  \,,
\end{equation}
where $\|\cdot\|_{R(\Real^3)}$ is the Rollnik norm.
This weaker version of Theorem~\ref{Thm.stability.3D.Kato}
is mentioned in~\cite[Rem.~6.2]{Kato_1966}
(see also~\cite[Thm.~III.12]{SiQF} and~\cite[Thm.~XIII.21]{RS4}
for partial results).
Note that $R(\Real^3) \supset L^{3/2}(\Real^3)$
by the Sobolev inequality.

An alternative approach was followed by Fanelli, Vega
and one of the present authors in~\cite{FKV}.
\begin{Theorem}[{\cite[Thm.~6.1]{Kato_1966}}, {\cite[Thm.~1]{FKV}}]%
\label{Thm.stability.3D}
Let $d = 3$ and $V \in L_\mathrm{loc}^{1}(\Real^d)$.
If there exists a constant $c<1$ such that
\begin{equation}\label{FKV}
  \forall \psi \in H^1(\Real^3)
  \,, \qquad
  \int_{\Real^3} |V| |\psi|^2 \leq c \int_{\Real^3} |\nabla\psi|^2
  \,,
\end{equation}
then the conclusions of Theorem~\ref{Thm.stability} hold true.
\end{Theorem}
\begin{proof}
First of all, notice that~\eqref{FKV} is equivalent to 
$
  \||V|^{1/2}H_0^{-1/2}g\|^2 
  \leq c \, \|g\|^2
$
for every $g \in \Ran(H_0^{1/2})$.
Since $0 \in \sigma_\mathrm{c}(H_0)$
(in fact, the spectrum of~$H_0$ is purely continuous),
the range $\Ran(H_0^{1/2})$ is dense in $\sii(\Real^3)$.
Consequently, 
$
  |V|^{1/2}H_0^{-1/2}
$
extends to a bounded operator in $\sii(\Real^3)$
with 
\begin{equation}\label{a1}
  \big\||V|^{1/2}H_0^{-1/2}\big\| \leq \sqrt{c}
  \,.
\end{equation}
By taking the adjoint, 
$
  H_0^{-1/2}|V|^{1/2}
$
also extends to a bounded operator in $\sii(\Real^3)$ with 
\begin{equation}\label{a2}
  \big\|H_0^{-1/2}|V|^{1/2}\big\| \leq \sqrt{c}
  \,.
 \end{equation}

We come back to the inequality~\eqref{Kolja}
valid for every $\phi,\psi \in C_0^\infty(\Real^3)$.
Using the dominated convergence theorem, 
we write
$$
\begin{aligned}
  |(\phi,K_z\psi)|
  &\leq 
  \lim_{\eps \to 0^+}
  \int_{\Real^3}\int_{\Real^3} 
  (|V|^{1/2}|\phi|)(x) \, G_{-\eps^2}(x,y) \, (|V|^{1/2}|\psi|)(y) 
  \, \der x \, \der y
  \\
  &=
  \lim_{\eps \to 0^+}
  \big(|V|^{1/2}|\phi|,(H_0+\eps^2)^{-1}|V|^{1/2}|\psi|\big)
  \\
  &=
  \lim_{\eps \to 0^+}
  \big((H_0+\eps^2)^{-1/2}|V|^{1/2}|\phi|,(H_0+\eps^2)^{-1/2}|V|^{1/2}|\psi|\big)
  \\
  &= 
  \big(H_0^{-1/2}|V|^{1/2}|\phi|,H_0^{-1/2}|V|^{1/2}|\psi|\big)
  \\
  &\leq 
  \big\|H_0^{-1/2}|V|^{1/2}\big\|^2 \|\phi\| \|\psi\|
  \\
  &\leq c \, \|\phi\| \|\psi\|
  \,.
\end{aligned}
$$
Here the last equality employs that 
$|V|^{1/2}|\phi|, |V|^{1/2}|\psi| \in \Ran(H_0^{1/2})$.
Since $C_0^\infty(\Real^3)$ is dense in $\sii(\Real^3)$,
we get $\|K_z\| \leq c$ uniformly in $z \in \Com\setminus(0,+\infty)$,
so~\eqref{Ass.small1} holds true.
Furthermore, the same estimates as above ensure 
that the supremum in~\eqref{eq:5} is 
bounded from above by the constant~$2c$. 
Hence, the sufficient conditions of the abstract
Theorem~\ref{Thm.main} and	 Corollary \ref{Cor3} are satisfied.  
\end{proof}
\begin{Remark} 
Except for the similarity of~$H_V$ and~$H_0$,
Theorem~\ref{Thm.stability.3D} was derived in~\cite{FKV}
without the knowledge of Kato's Theorem~\ref{Thm.main.similar} 
from~\cite{Kato_1966}.
Unaware of~Theorem~\ref{Thm.main0},
the inclusion $\sigma_\mathrm{c}(H_V) \subset \sigma(H_0)$
was derived by explicitly constructing a singular sequence of~$H_V$
corresponding to all points of $[0,+\infty)$.

It turns out that the hypotheses~\eqref{Kato} and~\eqref{FKV} are equivalent. 
The fact that~\eqref{FKV} implies~\eqref{Kato} is clear from the proof
of Theorem~\ref{Thm.stability.3D}. 
Conversely, $L=TT^*$ with $T:=|V|^{1/2}H_0^{-1/2}$,
so~\cite{Kato} implies $\|T\| \leq \sqrt{c}$,
which is~\eqref{a1} equivalent to~\eqref{FKV}. 
\end{Remark}

By the Sobolev inequality,
\eqref{FKV} holds provided that $V \in L^{3/2}(\Real^3)$
and (\cf~\eqref{Sobolev})
$$
  \|V\|_{L^{3/2}(\Real^3)} < C_{3/2,3} = 3^{3/2}\pi^2/4
  \,.
$$
This gives an estimate to the constant~$c_3$ of Theorem~\ref{Thm.stability}.
It turns out that this value is optimal
as demonstrated by Frank~\cite[Thm.~2]{Frank_2011}.
Outside the range of the Lebesgue as well as Rollnik classes,
sufficient conditions ensuring~\eqref{FKV} follow by the Hardy inquality
$-\Delta \geq (1/4)|x|^{-2}$, see~\cite[Eq.~(7)]{FKV}.
To conclude, let us compare the smallness sufficient conditions 
which ensure that the operators~$H_V$ and~$H_0$
are similar to each other in the three-dimensional situation: 
$$
\begin{aligned}
 & \eqref{Frank} 
  & \qquad \Longrightarrow \qquad &
  \eqref{Rollnik}
  & \qquad \Longrightarrow \qquad &
  \eqref{FKV}
  & \qquad \Longleftrightarrow \qquad &
  \eqref{Kato}
  \,.
  \\
  &\text{Lebesgue } L^{3/2} 
  && \text{Rollnik } R 
  && \text{form-subordination}
  && \text{Kato}
\end{aligned}
$$

We expect that Theorem~\ref{Thm.stability.3D} extends to higher dimensions.
\begin{Conjecture} 
Let $d > 3$ and $V \in L_\mathrm{loc}^{1}(\Real^d)$.
If there exists a constant $c<1$ such that
\begin{equation*} 
  \forall \psi \in H^1(\Real^d)
  \,, \qquad
  \int_{\Real^d} |V| |\psi|^2 \leq c \int_{\Real^d} |\nabla\psi|^2
  \,,
\end{equation*}
then the conclusions of Theorem~\ref{Thm.stability} hold true.
\end{Conjecture}
\subsubsection{Low dimensions}
The spectral stability 
does not hold in low dimensions $d=1,2$,
because of the criticality of the Laplacian when $d<3$.
Indeed, it is well known 
(see, \eg, \cite[Thm.~XIII.11]{RS4})
that $H_V$~possesses at least one
(discrete) negative eigenvalue whenever $V \in C_0^\infty(\Real^d)$
is real-valued, non-positive and non-trivial and $d=1,2$.
In dimension $d=2$, however, the spectral stability can be achieved
by adding a magnetic field to~$H_0$, see~\cite{FKV2}.

In any case, the Birman--Schwinger principle 
can be used to obtain sharp estimates for the eigenvalues,
even when~$V$ is complex-valued. 
Here we focus on dimension $d=1$, 
where a simple formula
for the integral kernel of the resolvent of~$H_0$ is available.

\begin{Theorem}[{\cite[Thm.~4]{Abramov-Aslanyan-Davies_2001}} \&
{\cite[Corol.~2.16]{Davies-Nath_2002}}]%
\label{Thm.Davies}
Let $d = 1$ and $V \in L^{1}(\Real)$.
\begin{enumerate}
\item[\emph{(i)}]
$\sigma_\mathrm{r}(H_V) = \emptyset$.
\item[\emph{(ii)}]
$\sigma_\mathrm{e}(H_V) = [0,+\infty)$.
\item[\emph{(iii)}]
$
  \sigma_\mathrm{p}(H_V) \subset 
  \big\{\lambda \in \Com: |\lambda| \leq \frac{1}{4}\|V\|_{L^1(\Real)}^2\big\}
$.
\end{enumerate}
\end{Theorem}
\begin{proof}
Property~(i) is a general fact for Schr\"odinger operators
because of the $\mathcal{T}$-self-adjointness property 
$H_V^*=\mathcal{T}H_V\mathcal{T}$, where $\mathcal{T}\psi:=\bar{\psi}$
is the complex conjugation (time-reversal operator in quantum mechanics).
Consequently, if~$\bar{\lambda}$ is an eigenvalue of~$H_V^*$,
then necessarily~$\lambda$ is an eigenvalue of~$H_V$,
so~(i) follows by the general criterion~\eqref{residual}.  

The other properties employ the fact that
the unperturbed resolvent $(H_0-z)^{-1}$
is an integral operator in~$\sii(\Real)$ with the kernel
$$
  G_z(x,y) := \frac{e^{-\sqrt{-z}\,|x-y|}}{2\,\sqrt{-z}}
  \,,
$$
where $z \in \Com\setminus[0,+\infty)$.
Consequently,  
$$
  \forall z \in \Com\setminus[0,+\infty), \
  x,y \in \Real \,, \qquad
  |G_z(x,y)| = \frac{e^{-\Re\sqrt{-z}\,|x-y|}}{2\,|\sqrt{-z}|}	
  \leq \frac{1}{2\, \sqrt{|z|}} 
$$

Property~(ii) follows because of the compactness of~$K_z$.
Under the hypotheses $V \in L^1(\Real)$, 
the operator~$H_V$ is m-sectorial
(\cf~Remark~\ref{Rem.Friedrichs}).  
Hence, there exists a negative~$z$ with sufficiently large~$|z|$
such that $z \in \rho(H_0)\cap\rho(H_V)$
and $(H_0-z)^{-1}$ is m-accretive.
Then
$$
  \|K_{z}\|_{\mathrm{HS}}^2
  = \int_{\Real} \int_{\Real} 
  |V(x)| \, |G_z(x,y)|^2 \,  |V(y)| 
  \, \der x \, \der y
  \leq \frac{\|V\|_{L^1(\Real)}^2}{4\,|z|}
  \,, 
$$
where $\|\cdot\|_{\mathrm{HS}}$ denotes the Hilbert--Schmidt norm,
so $K_{z}$~is compact. 
By Proposition~\ref{Prop.2nd},
$$
  (H_V-z)^{-1} - (H_0-z)^{-1}
  = -[\sgn \bar{V} \, |V|^{1/2}(H_0-z)^{-1}]^* |V|^{1/2} (H_V-z)^{-1}
  \,.
$$
Since $|V|^{1/2}(H_V-z)^{-1}$, $\sgn \bar{V}$ and $(H_0-z)^{-1/2}$  
are bounded operators, 
the difference of the resolvents
is compact if the operator $T := |V|^{1/2}(H_0-z)^{-1/2}$ is compact.
This is the case if, and only if, $TT^*$ is compact.
It remains to notice that $\|TT^*\|_\mathrm{HS}=\|K_z\|_\mathrm{HS}$ 
and recall the general stability theorem \cite[Thm.~IX.2.4]{Edmunds-Evans}.

Property~(iii) is the main part of the theorem. 
Similarly as above, we have
$$
  \|K_{z}\|^2
  \leq \|K_{z}\|_{\mathrm{HS}}^2
  \leq \frac{\|V\|_{L^1(\Real)}^2}{4 \, |z|}
$$
for every $z \in \Com\setminus[0,+\infty)$.
Consequently, $\|K_{z}\|>1$ if $|z| > \frac{1}{4}\|V\|_{L^1(\Real)}^2$.
This proves the desired inclusion (including the embedded eigenvalues)
by virtue of Corollary~\ref{Corol.evs}.
\end{proof}

The same machinery has been recently applied to possibly non-self-adjoint
biharmonic Schr\"odinger operators in~\cite{IKL}
and the wave operator with complex-valued dampings in~\cite{KK8}.
The Birman--Schwinger principle is not limited to continuous spaces;
see~\cite{Ibrogimov-Stampach_2019,Bogli-Stampach} 
for Schr\"odinger operators on lattices.

\subsection{Dirac operators in the three-dimensional Euclidean space}
Let $H_0 := - i \;\! \alpha \cdot \nabla  + m \, \alpha_4$
in $\Hilbert := \sii(\Real^3;\Com^4)$
with $\Dom(H_0) := H^1(\Real^3;\Com^4)$, 
where $m>0$ is a constant
and $\alpha := (\alpha_1,\alpha_2,\alpha_3)$
with~$\alpha_\mu$ being the usual $4 \times 4$ Hermitian Dirac matrices 
satisfying the anticommutation rules
$
  \alpha_\mu \alpha_\nu + \alpha_\nu \alpha_\mu
  = 2 \delta_{\mu\nu} I_{\Com^4} 
$
for $\mu,\nu\in\{1,\dots,4\}$ 
and the dot denotes the scalar product in~$\Real^3$.
One has $\sigma(H_0)=(-\infty,-m]\cup[+m,+\infty)$
and the spectrum is purely absolutely continuous.

Notice that $H_0^2 = (-\Delta+m^2) I_{\Com^4}$, 
where $-\Delta+m^2$ is the self-adjoint Schr\"odinger operator
in $\sii(\Real^3)$ with the usual domain $H^2(\Real^3)$.
The absolute value of~$H_0$ thus equals 
$|H_0| = \sqrt{-\Delta+m^2} I_{\Com^4}$,
which is again a self-adjoint operator 
when considered on the domain $H^1(\Real^3;\Com^4)$.
The form domain of $\sqrt{-\Delta+m^2}$ equals 
the fractional Sobolev space $H^{1/2}(\Real^3)$,
\cf~\cite[Sec.~7.11]{LL}.
Notice that $C_0^\infty(\Real^3)$ is dense in $H^{1/2}(\Real^3)$,
\cf~\cite[Sec.~7.14]{LL}. 

Given any $V \in L_\mathrm{loc}^1(\Real^3;\Com^{4 \times 4})$,
we use the matrix polar decomposition
$$
  V(x) = U(x) \, |V(x)| = U(x) \, |V(x)| \, |V(x)|^{1/2}
$$
for almost every $x \in \Real^3$.
Here~$U(x)$ is unitary
and $|V(x)|=\sqrt{V(x)^*V(x)}$ as before.
We set $A(x):=|V(x)|^{1/2}$ and $B(x) := |V(x)|^{1/2} U(x)^*$
as in the case of Schr\"odinger operators.
Now, however, we have $A(x) U(x)^* \not= U(x)^* A(x)$ in general,
which somewhat complicates the analysis.
We use the same symbols $A,B$ for the extended operators of 
matrix multiplication 
initially defined on $\mathcal{D}:=C_0^\infty(\Real^3;\Com^4)$.
Notice that $\mathcal{D}$ is dense in 
$H^{1/2}(\Real^3;\Com^4) = \Dom(|H_0|^{1/2})$.
 
To minimise conditions imposed on the matrix-valued potential~$V$, 
we follow~\cite{FK9} and
consider the matrix norm $v(x) := \|V(x)\|_{\Com^4\to\Com^4}$
for almost every $x \in \Real^3$.
The non-negative scalar function~$v$ belongs to $L_\mathrm{loc}^1(\Real^3)$.
Note that 
$ 
  v(x) = \| |V(x)| \|_{\Com^4\to\Com^4} 
  = \| |V(x)|^{1/2} \|_{\Com^4\to\Com^4}^2
$
We assume that there exist numbers $a \in (0,1)$ and $b\in\Real$ such that
\begin{equation}\label{Ass.Dirac}
  \forall f\in C_0^\infty(\Real^3)
  \,, \qquad
  \int_{\Real^3} v(x) |f(x)|^2 \, \der x
  \leq a \int_{\Real^3} |\sqrt[4]{-\Delta}\,f(x)|^2 \, \der x
  + b \int_{\Real^3} |f(x)|^2 \, \der x
  \,.
\end{equation}
Then Assumption~\ref{Ass.Ass} holds true. 
A sufficient condition to satisfy~\eqref{Ass.Dirac} is $v = v_1 + v_2$ 
with $v_1 \in L^3(\Real^3)$ and $v_2 \in L^\infty(\Real^3)$,
where
$
  \|v_1\|_{L^3(\Real^3)} < (2\pi^2)^{1/3}
$.
This can be shown  
with help of the H\"older inequality
and a quantified version of the Sobolev-type embedding 
$\dot{H}^{1/2}(\Real^3) \hookrightarrow L^3(\Real^3)$,
see~\cite[Prop.~1]{FK9}. 
Alternative sufficient conditions can be obtained 
by means of Kato's inequality
$
  \sqrt{-\Delta} \geq (2/\pi) |x|^{-1}
$,
see \cite[Rem.~3]{FK9}. 

In summary, $V$~falls within the class
of perturbations considered in Assumption~\ref{Ass.Ass}
and the pseudo-Friedrichs extension~$H_V$ is well defined. 
Contrary to Schr\"odinger operators, the Dirac operators
cannot be introduced via the Friedrichs extension
because of the unboundedness from below of the latter.

To apply the Birman--Schwinger principle to~$H_V$,
one customarily uses the identity 
$(H_0-z)^{-1} = (H_0+z) (H_0^2-z^2)^{-1}$
to get an explicit formula for the unperturbed resolvent.
More specifically, $(H_0-z)^{-1}$ is an integral operator in~$\Hilbert$ 
with the kernel obtained by applying~$H_0+z$ to 
the Green function~\eqref{Green} at energy $m^2-z^2$.
Estimating the norm of~$K_z$ by the Hilbert--Schmidt norm
and applying Corollary~\ref{Corol.evs},
one obtains various enclosures for the eigenvalues of~$H_V$.
This strategy was followed by Fanelli and one of the present authors in~\cite{FK9}.
As an example, we mention the following result.

\begin{Theorem}[{\cite[Thm.~2]{FK9}}]\label{Thm2}
Assume $v \in L^3(\Real^3) \cap L^{3/2}(\Real^3)$.
If  
\begin{equation}\label{hypothesis2}
  C_1 \, \|v\|_{L^3(\Real^3)} 
  + C_2 \, |\Re \lambda| \,
  \|v\|_{L^{3/2}(\Real^{3})} 
  < 1
  \,,
\end{equation}
where 
$$
  C_1 := \left(\frac{\pi}{2}\right)^{1/3} \sqrt{1+e^{-1}+2e^{-2}}
  \qquad \mbox{and} \qquad
  C_2 := \frac{2^{17/6}}{3\pi^{2/3}}
  \,,
$$
then $\lambda \not\in \sigma_\mathrm{p}(H_V)$.
\end{Theorem}

Given a potential~$V$ with sufficiently small norm $\|v\|_{L^3(\Real^3)}$,
the hypothesis~\eqref{hypothesis2} excludes the existence of eigenvalues 
in thin tubular neighbourhoods of the imaginary axis,
with the thinness determined by the norm $\|v\|_{L^{3/2}(\Real^3)}$. 
Note that eigenvalues embedded in 
the essential spectrum $(-\infty,-m]\cup[+m,+\infty)$
are excluded as well. 
As an alternative result, \cite[Thm.~1]{FK9}~provides a quantitative enclosure
for more general potentials satisfying merely $v \in L^3(\Real^3)$.

For~$V$ being matrix-valued and possibly non-Hermitian,
a systematic study of the spectrum of the Dirac operator~$H_V$ was initiated
by the pioneering work of Cuenin, Laptev and Tretter 
\cite{Cuenin-Laptev-Tretter_2014} in the one-dimensional setting
and followed by \cite{Cuenin_2014,Enblom_2018,Cuenin-Siegl_2018}.
Some spectral aspects in the present three-dimensional 
situation are also covered by
\cite{Dubuisson_2014,Sambou_2016,Cuenin_2017,CFK,
D'Ancona-Fanelli-Schiavone}.
The same machinery has been recently applied to non-self-adjoint
Dirac operators on lattices \cite{CIKS}.

\subsection{Schr\"odinger operators in three-dimensional hyperbolic space}
In order to derive completely new results
with the help of the Birman--Schwinger principle,
we eventually consider Schr\"odinger operators
in hyperbolic spaces.
This class of operators does not seem to have been
considered previously in the non-self-adjoint context
except for the recent works \cite{Chen_2018,Hansmann_2019}.
However, the study of spectral properties of self-adjoint realisations
is enormous, see, \eg, 
\cite{Hislop_1994,Levin-Solomyak_1997,
Karp-Peyerimhoff_2000,Benguria-Linde_2007,
Borthwick-Crompton_2014,Berchio-Ganguly-Grillo-Pinchover_2019}
and references therein.
Here we restrict ourselves to the three-dimensional case
and refer to~\cite{Hansmann_2019} and~\cite{Chen_2018}
for the hyperbolic plane and higher dimensions, respectively.

Let $\Hyper^3$ be the three-dimensional hyperbolic space,
\ie~a complete, simply connected Riemannian manifold
with all sectional curvatures equal to~$-1$. 
There are three (isometric) standard realisations of~$\Hyper^3$
given by the half-space, ball and hyperboloid models
(\cf~\cite[Sec.~1]{Hislop_1994}), but we shall not need them.
We denote by~$H_0$ the self-adjoint Laplacian in $\Hilbert:=\sii(\Hyper^3)$,
introduced in a standard way
as the Friedrichs extension of the Laplace--Beltrami operator
initially defined on $\mathcal{D} := C_0^\infty(\Hyper^3)$.
More specifically, $H_0$~is the operator associated 
with the closed form $h_0[\psi] := \int_{\Hyper^3} |\nabla\psi|^2$
with
$
  \Dom(h_0) := H^1(\Hyper^3)
$
being the usual Sobolev space.
The absolute value~$|H_0|$ satisfies 
$\||H_0|^{1/2}\psi\|=\|\nabla\psi\|$
for every $\psi \in \Dom(|H_0|^{1/2})= H^1(\Hyper^3)$.
Note that $C_0^\infty(\Hyper^3)$ is  a core of $|H_0|^{1/2}$.
It is well known~\cite[Sec.~2]{Hislop_1994} that
\begin{equation}\label{spec0}
  \sigma(H_0) = [1,+\infty)
\end{equation}
and that the spectrum is purely absolutely continuous.
The shifted operator $H_0-1$ is \emph{subcritical},
meaning that it satisfies a Hardy-type inequality 
(see \cite{Akutagawa-Kumura_2013,Berchio-Ganguly-Grillo_2017}
for original proofs and \cite{Berchio-Ganguly-Grillo-Pinchover_2019}
for recent	 improvements)
\begin{equation}\label{Hardy.hyper}
  \int_{\mathbb{H}^3} |\nabla\psi|^2 - \int_{\mathbb{H}^3} |\psi|^2
  \geq \frac{1}{4} \int_{\mathbb{H}^3}
  \frac{|\psi(x)|^2}{\rho(x,x_0)^2} \, \der x
  \,,
\end{equation}
where $\rho(x,x_0)$ denotes the Riemannian distance between
the points $x,x_0 \in \mathbb{H}^3$ and~$x_0$ is fixed. 

Now let $V \in L_\mathrm{loc}^1(\Hyper^3)$ 
and make the same decomposition~\eqref{decomposition} 
as in the Euclidean case.
The operators~$A,B$ are defined analogously. 
We assume the subordination condition
\begin{equation}\label{Ass.hyper}
  \exists c < 1 
  \,, \qquad 
  \forall \psi \in H^1(\mathbb{H}^3)
  \,, \qquad  
  \int_{\mathbb{H}^3} |V| |\psi|^2 
  \leq c \left(
  \int_{\mathbb{H}^3} |\nabla\psi|^2 - \int_{\mathbb{H}^3} |\psi|^2
  \right)
  \,.
\end{equation}
Then Assumption~\ref{Ass.Ass} holds true 
and the pseudo-Friedrichs extension~$H_V$ is well defined.
It coincides with the usual m-sectorial Friedrichs extension in this case,
because~\eqref{Ass.hyper} ensures that~$V$ is relatively form-bounded 
with respect to~$H_0$ with the relative bound less than~$1$.  
In view of~\eqref{Hardy.hyper}, a sufficient condition 
to satisfy~\eqref{Ass.hyper} is given by the pointwise inequality
$
  |V(x)| \leq (c/4) \rho(x,x_0)^{-2}
$
for almost every $x \in \Hyper^3$.

We note that~\eqref{Ass.hyper} implies that the shifted operator 
$H_V-1$ is m-accretive and hence, in particular, 
the spectrum of~$H_V$ is contained in the complex half-plane 
$\{ \lambda : \Re\lambda \geq 1\}$. 
Actually, a much stronger statement is true.

\begin{Theorem}\label{Thm.hyper}
If~\eqref{Ass.hyper} holds, then
$$
  \sigma(H_V) = \sigma_\mathrm{c}(H_V) =[1,+\infty)
  \,.
$$
Moreover, $H_V$ and $H_0$ are similar to each other.
\end{Theorem}
\begin{proof}
The proof is similar to the proof of Theorem~\ref{Thm.stability.3D}.
We start with an equivalent formulation of~\eqref{Ass.hyper}. 
Writing $g := (H_0-1)^{1/2} \psi$ in~\eqref{Ass.hyper}, we have
\begin{equation*} 
  \big\||V|^{1/2}(H_0-1)^{-1/2} g\big\|^2 
  \leq c \left(
  \big\|\nabla (H_0-1)^{-1/2} g\big\|^2  
  - \big\|(H_0-1)^{-1/2} g\big\|^2 
  \right)
  = c \, \|g\|^2
  \,.
\end{equation*}
Since $1 \in \sigma_\mathrm{c}(H_0)$ 
(in fact, the spectrum of~$H_0$ is purely continuous),
the range of~$(H_0-1)^{1/2}$ is dense in $\sii(\Hyper^3)$
and we see that~\eqref{eq:2} is equivalent to
\begin{equation}\label{b1} 
  \big\||V|^{1/2}(H_0-1)^{-1/2}\big\|^2 
  \leq c  
  \,.
\end{equation}
It follows (by taking the adjoint) that also
\begin{equation}\label{b2}
  \big\|(H_0-1)^{-1/2}|V|^{1/2}\big\|^2 
  \leq c 
  \,.
\end{equation} 

The main ingredient of the proof is the explicit form
of the integral kernel~$G_z(x,y)$ of the unperturbed 
resolvent $(H_0-z)^{-1}$ which is given by
\begin{equation}\label{Green.hyper}
  G_z(x,y) := \frac{e^{-\sqrt{-(z-1)}\,\rho(x,y)}}{4\pi \sinh\rho(x,y)} 
  \,,
\end{equation}
where $z \in \Com\setminus(1,+\infty)$
and $x,y \in \Hyper^3$ with $x \not= y$.
To get~\eqref{Green.hyper}, one may integrate the formula 
for the heat kernel~\cite[p.~179]{Davies_1989}
over positive times.
As in the Euclidean case (\cf~\eqref{peculiar}),
one has the uniform pointwise bound
\begin{equation*}
  \forall z \not\in (1,+\infty)
  \,, \quad 
  \forall x,y \in \Hyper^3, \ x\not=y
  \,, \qquad 
  |G_z(x,y)| \leq G_1(x,y)
  \,.
\end{equation*}
Consequently, 
for every $\phi,\psi \in C_0^\infty(\Hyper^3)$, one has
%
\begin{align*}
  |(\phi,K_z\psi)|
  &\leq 
  \int_{\Hyper^3}\int_{\Hyper^3} 
  (|V|^{1/2}|\phi|)(x) \, |G_z(x,y)| \, (|V|^{1/2}|\psi|)(y) 
  \, \der x \, \der y
  \\
  &\leq  
  \int_{\Hyper^3}\int_{\Hyper^3} 
  (|V|^{1/2}|\phi|)(x) \, G_1(x,y) \, (|V|^{1/2}|\psi|)(y) 
  \, \der x \, \der y
  \\
  &= \lim_{\eps \to 0^+}
  \int_{\Hyper^3}\int_{\Hyper^3} 
  (|V|^{1/2}|\phi|)(x) \, G_{1-\eps^2}(x,y) \, (|V|^{1/2}|\psi|)(y) 
  \, \der x \, \der y
  \\
  &= \lim_{\eps \to 0^+}
  \big(|V|^{1/2}|\phi|,(H_0-1+\eps^2)^{-1}|V|^{1/2}|\psi|\big)
  \\
  &= \lim_{\eps \to 0^+}
  \big((H_0-1+\eps^2)^{-1/2}|V|^{1/2}|\phi|,(H_0-1+\eps^2)^{-1/2}|V|^{1/2}|\psi|\big)
  \\
  &=  
  \big((H_0-1)^{-1/2}|V|^{1/2}|\phi|,(H_0-1)^{-1/2}|V|^{1/2}|\psi|\big)
  \\
  &\leq 
  \big\|(H_0-1)^{-1/2}|V|^{1/2}\big\|^2 \|\phi\| \|\psi\|
  \\
  &\leq c \, \|\phi\| \|\psi\|
  \,.
\end{align*}
%
Here the limits are justified with help of the dominated convergence theorem
and the last inequality follows by~\eqref{b2}.  
Since $C_0^\infty(\Hyper^3)$ is dense in $\sii(\Hyper^3)$,
we get $\|K_z\| \leq c$ uniformly in $z \in \Com\setminus[1,+\infty)$,
so~\eqref{Ass.small1} holds true.
Furthermore, the same estimates as above ensure 
that the supremum in~\eqref{eq:5} 
is bounded from above by the constant~$2c$.
Hence, the sufficient conditions of the abstract 
Theorem~\ref{Thm.main} and Corollary \ref{Cor3} are satisfied.
\end{proof}  

\appendix
\section{Kato's and pseudo-Friedrichs extensions coincide}  
Suppose that $H_0, A, B$ satisfy Assumption \ref{Ass.Ass} and that in addition $A,B$ are closed and smooth relative to $H_0$. Moreover, suppose that 
$$ \Dom(A)= \Dom(B)= \Dom(|H_0|^{1/2}).$$
Let $H_V$ denote the pseudo-Friedrichs extension constructed in Section \ref{Sec.pseudo} and let $\tilde{H}_V$ denote the closed extension of $H_0+B^*A$ provided by Kato's Theorem \ref{Thm.main.similar}.  

\begin{Proposition}\label{Prop.Appendix} 
Given the above assumptions we have $H_V=\tilde{H}_V$. 
\end{Proposition}
\begin{proof} 
By \cite[Theorem 1.5]{Kato_1966}, $A$ is smooth relative to $\tilde{H}_V$ and $B$ is smooth relative to $\tilde{H}_V^*$, hence $\Dom(\tilde{H}_V) \subset \Dom(A) = \Dom(|H_0|^{1/2})$
and $\Dom(\tilde{H}_V^*) \subset \Dom(B) = \Dom(|H_0|^{1/2})$,
which establishes two of the uniqueness requirements of Theorem~\ref{Thm.pseudo}.
It remains to verify~\eqref{form} and~\eqref{form*}.  
Let $\phi \in \Dom(|H_0|^{1/2})$ and $\psi \in \Dom(\tilde{H}_V)$.
Given $\xi \in \Com\setminus\Real$, 
let $g \in \Hilbert$ be the unique vector satisfying
$\psi = (\tilde{H}_V-\xi)^{-1}g$. 
Then, using~\eqref{2nd},
$$
\begin{aligned}
  h_0(\phi,\psi)
  &= \big(G_0^{1/2}\phi,H_0 G_0^{-1} G_0^{1/2}\psi\big) 
  = \big(G_0^{1/2}\phi,H_0 G_0^{-1/2} (\tilde{H}_V-\xi)^{-1}g\big) 
  \\
  &= \big(G_0^{1/2}\phi,H_0 G_0^{-1/2} (H_0-\xi)^{-1}g\big) 
  - \big(G_0^{1/2}\phi,
  H_0 G_0^{-1/2} \, \overline{(H_0-\xi)^{-1}B^*} A (\tilde{H}_V-\xi)^{-1}g\big) 
  \\
  &= (\phi,g) + \xi \big(\phi,(H_0-\xi)^{-1}g\big)
  - \big(G_0^{1/2}\phi,
  H_0 G_0^{-1/2} \, \overline{(H_0-\xi)^{-1}B^*} A \psi\big) 
  \,.
\end{aligned}
$$
If $\phi \in \Dom(H_0)$, then
$$
\begin{aligned}
  \big(G_0^{1/2}\phi,H_0 G_0^{-1/2} \, \overline{(H_0-\xi)^{-1}B^*} A \psi\big)
  &= \big(H_0\phi,\overline{(H_0-\xi)^{-1}B^*} A \psi\big)
  \\
  &= \big([(H_0-\xi)^{-1}B^*]^*H_0\phi, A \psi\big)
  \\
  &= \big(B(H_0-\bar{\xi})^{-1}H_0\phi, A \psi\big)
  \\
  &=
  (B\phi,A\psi) + \xi \big(B(H_0-\bar{\xi})^{-1}\phi,A\psi\big)
  \\
  &=
  v(\phi,\psi) + \xi \big(\phi,\overline{(H_0-\xi)^{-1}B^*} A \psi\big)
  \,,
\end{aligned}
$$
where we used that 
$
  \overline{(H_0-\xi)^{-1}B^*} = [(H_0-\xi)^{-1}B^*]^{**} = [B(H_0-\overline{\xi})^{-1}]^*
$
which follows from the fact that $B^*$ is densely defined (since $B$ is closed) and $[(H_0-\xi)^{-1}B^*]^* = B(H_0-\overline{\xi})^{-1} \in \mathscr{B}(\Hilbert,\Hilbert')$ is densely defined as well. The obtained identity extends to all $\phi \in \Dom(|H_0|^{1/2})$,
since $\Dom(H_0)$ is a core of $\Dom(|H_0|^{1/2})$.
Therefore, using~\eqref{2nd} again,
$$
\begin{aligned}
  h_V(\phi,\psi) &= h_0(\phi,\psi) + v(\phi,\psi)
  \\
  &= (\phi,g) + \xi \big(\phi,(H_0-\xi)^{-1}g\big)
  -\xi \big(\phi,\overline{(H_0-\xi)^{-1}B^*}A\psi\big)
  \\
  &= (\phi,g) + \xi \big(\phi,(\tilde{H}_V-\xi)^{-1}g\big) \\
  &= (\phi, (\tilde{H}_V-\xi)\psi) + \xi (\phi, \psi) = (\phi, \tilde{H}_V \psi)
\end{aligned}
$$
for every $\phi \in \Dom(|H_0|^{1/2})$ 
and $\psi \in \Dom(\tilde{H}_V)$.
This establishes~\eqref{form}. The validity of~\eqref{form*} can be proved in the same manner.
The uniqueness of the pseudo-Friedrichs extension 
ensures that necessarily $\tilde{H}_V=H_V$ as desired. 
\end{proof}

\subsection*{Acknowledgments}
We are grateful to Yehuda Pinchover for useful discussions.
M.H. was supported by the Deutsche Forschungsgemeinschaft (DFG, German Research Foundation) - Project number HA 7732/2-2. 
The work of D.K. was partially supported 
by the GA\v{C}R (Czech Science Foundation)
grants No.~18-08835S and 20-17749X.
 
%
\bibliography{bib07}
\bibliographystyle{amsplain}
\end{document}